\documentclass[12pt, reqno, twoside, letterpaper]{amsart}

\usepackage{amsmath,amssymb,amsbsy,amsfonts,amsthm,latexsym,
	amsopn,amstext,amsxtra,euscript,amscd,stmaryrd,mathrsfs,
	cite,array}

\numberwithin{equation}{section}

\usepackage{color}
\usepackage[usenames,dvipsnames,svgnames,table]{xcolor}

\def\eps{\varepsilon}
\def\mand{\qquad\mbox{and}\qquad}
\def\fl#1{\left\lfloor#1\right\rfloor}

\def\({\left(}
\def\){\right)}

\newcommand{\e}{\ensuremath{\mathbf{e}}}

\newcommand{\cB}{\ensuremath{\mathcal{B}}}

\newcommand{\cS}{\ensuremath{\mathcal{S}}}

\newcommand{\cT}{\ensuremath{\mathcal{T}}}

\newcommand{\cX}{\ensuremath{\mathcal{X}}}

\newcommand{\sL}{\ensuremath{\mathscr{L}}}

\newcommand{\sN}{\ensuremath{\mathscr{N}}}

\newcommand{\NN}{\ensuremath{\mathbb{N}}}
\newcommand{\PP}{\ensuremath{\mathbb{P}}}

\newcommand{\RR}{\ensuremath{\mathbb{R}}}
\newcommand{\ZZ}{\ensuremath{\mathbb{Z}}}

\usepackage{color}
\usepackage{hyperref}
\hypersetup{
	pdftoolbar=true,
	pdfmenubar=false,
	pdffitwindow=false,
	pdfstartview={FitH},
	pdftitle={Piatetski-Shapiro Primes in the intersection of multiple Beatty sequences}, 
	pdfauthor={Victor Zhenyu Guo, Jinjiang Li, Min Zhang}, 
	pdfcreator={Victor Zhenyu Guo, Jinjiang Li, Min Zhang}, 
	pdfproducer={Victor Zhenyu Guo, Jinjiang Li, Min Zhang}, 
	pdfnewwindow=true,
	colorlinks=true,
	linkcolor={black},
	citecolor={black},
	filecolor={black},
	urlcolor={black},
}

\usepackage[text={144mm,233mm},centering]{geometry} 
\usepackage{setspace}
\usepackage{amsthm}
\newtheoremstyle{customthm}
{1em}                    
{1em}                    
{\itshape}               
{}                       
{\scshape}               
{.}                      
{5pt plus 1pt minus 1pt} 
{}                       

\newtheoremstyle{customrem}
{1em}                    
{1em}                    
{}                       
{}                       
{\scshape}               
{.}                      
{5pt plus 1pt minus 1pt} 
{}                       

\theoremstyle{customthm}

\newtheorem{X}{X}[section]
\newtheorem{theorem}[X]{Theorem}

\newtheorem{lemma}[X]{Lemma}

\theoremstyle{customrem}

\usepackage{etoolbox}
\AtEndEnvironment{remark}{\null\hfill\qedsymbol}

\AtEndEnvironment{definition}{\null\hfill\qedsymbol}

\renewcommand{\le}{\ensuremath{\leqslant}}
\renewcommand{\ge}{\ensuremath{\geqslant}}
\def\fl#1{\left\lfloor#1\right\rfloor}

\renewcommand{\pod}[1]{\mathchoice
	{\allowbreak \if@display \mkern 5mu\else \mkern 5mu\fi (#1)}
	{\allowbreak \if@display \mkern 5mu\else \mkern 5mu\fi (#1)}
	{\mkern4mu(#1)}
	{\mkern4mu(#1)}
}

\newcommand*{\defeq}{\mathrel{\vcenter{\baselineskip0.5ex \lineskiplimit0pt
			\hbox{\scriptsize.}\hbox{\scriptsize.}}}%
	=}

\DeclareSymbolFont{EUEX}{U}{euex}{m}{n}

\DeclareSymbolFont{euexlargesymbols}{U}{euex}{m}{n}
\DeclareMathSymbol{\intop}{\mathop}{euexlargesymbols}{"52}
\def\int{\intop\nolimits}

\DeclareSymbolFont{euexsymbols}     {U}{euex}{m}{n}
\DeclareMathSymbol{\smallint}{\mathop}{euexsymbols}{"52}

\usepackage{ifthen}
\usepackage{xifthen}

\def\sums{
	\@ifnextchar[
	{\sums@i}
	{\ensuremath{\sum}}
}
\def\sums@i[#1]{
	\@ifnextchar[
	{\sums@ii{#1}}
	{\ensuremath{\sum_{#1}}}
}
\def\sums@ii#1[#2]{
	\@ifnextchar[
	{\sums@iii{#1}{#2}}
	{\ensuremath{\sum_{\substack{#1 \\ #2}}}}
}
\def\sums@iii#1#2[#3]{
	\@ifnextchar[
	{\sums@iv{#1}{#2}{#3}}
	{\ensuremath{\sum_{\substack{#1 \\ #2 \\ #3}}}}
}
\def\sums@iv#1#2#3[#4]{
	\@ifnextchar[
	{\sums@v{#1}{#2}{#3}{#4}}
	{\ensuremath{\sum_{\substack{#1 \\ #2 \\ #3 \\ #4}}}}
}
\def\sums@v#1#2#3#4[#5]{
	{\ensuremath{\sum_{\substack{#1 \\ #2 \\ #3 \\ #4 \\ #5}}}}
}

\def\sumss[#1]{
	\@ifnextchar[
	{\sumss@i[#1]}
	{
		\ifthenelse{\isempty{#1}}
		{\ensuremath{\sum}}
		{
			\ifthenelse{\equal{#1}{'}}
			{\ensuremath{\sideset{}{^{\prime}}{\sum}}}
			{\ensuremath{\sideset{}{^{#1}}{\sum}}}
		}
	}
}
\def\sumss@i[#1][#2]{
	\@ifnextchar[
	{\sumss@ii[#1]{#2}}
	{
		\ifthenelse{\isempty{#1}}
		{\ensuremath{\sum_{#2}}}
		{
			\ifthenelse{\equal{#1}{'}}
			{\ensuremath{\sideset{}{^{\prime}}{\sum}_{#2}}}
			{\ensuremath{\sideset{}{^{#1}}{\sum}_{#2}}}
		}
	}
}
\def\sumss@ii[#1]#2[#3]{
	\@ifnextchar[
	{\sumss@iii[#1]{#2}{#3}}
	{
		\ifthenelse{\isempty{#1}}
		{\ensuremath{\sum_{\substack{#2 \\ #3}}}}
		{
			\ifthenelse{\equal{#1}{'}}
			{\ensuremath{\sideset{}{^{\prime}}{\sum}_{\substack{#2 \\ #3}}}}
			{\ensuremath{\sideset{}{^{#1}}{\sum}_{\substack{#2 \\ #3}}}}
		}
	}
}
\def\sumss@iii[#1]#2#3[#4]{
	\@ifnextchar[
	{\sumss@iv[#1]{#2}{#3}{#4}}
	{
		\ifthenelse{\isempty{#1}}
		{\ensuremath{\sum_{\substack{#2 \\ #3 \\ #4}}}}
		{
			\ifthenelse{\equal{#1}{'}}
			{\ensuremath{\sideset{}{^{\prime}}{\sum}_{\substack{#2 \\ #3 \\ #4}}}}
			{\ensuremath{\sideset{}{^{#1}}{\sum}_{\substack{#2 \\ #3 \\ #4}}}}
		}
	}
}
\def\sumss@iv[#1]#2#3#4[#5]{
	\@ifnextchar[
	{\sumss@v[#1]{#2}{#3}{#4}{#5}}
	{
		\ifthenelse{\isempty{#1}}
		{\ensuremath{\sum_{\substack{#2 \\ #3 \\ #4 \\ #5}}}}
		{
			\ifthenelse{\equal{#1}{'}}
			{\ensuremath{\sideset{}{^{\prime}}{\sum}_{\substack{#2 \\ #3 \\ #4 \\ #5}}}}
			{\ensuremath{\sideset{}{^{#1}}{\sum}_{\substack{#2 \\ #3 \\ #4 \\ #5}}}}
		}
	}
}
\def\sumss@v[#1]#2#3#4#5[#6]{
	{\ifthenelse{\isempty{#1}}
		{\ensuremath{\sum_{\substack{#2 \\ #3 \\ #4 \\ #5 \\ #6 }}}}
		{
			\ifthenelse{\equal{#1}{'}}
			{\ensuremath{\sideset{}{^{\prime}}{\sum}_{\substack{#2 \\ #3 \\ #4 \\ #5 \\ #6 }}}}
			{\ensuremath{\sideset{}{^{#1}}{\sum}_{\substack{#2 \\ #3 \\ #4 \\ #5 \\ #6 }}}}
		}
	}
}

\def\sumsstxt[#1]{
	\@ifnextchar[
	{\sumsstxt@i[#1]}
	{
		\ifthenelse{\isempty{#1}}
		{\ensuremath{\textstyle\sum}}
		{
			\ifthenelse{\equal{#1}{'}}
			{\ensuremath{\sideset{}{^{\prime}}{\textstyle\sum}}}
			{\ensuremath{\sideset{}{^{#1}}{\textstyle\sum}}}
		}
	}
}
\def\sumsstxt@i[#1][#2]{
	\@ifnextchar[
	{\sumsstxt@ii[#1]{#2}}
	{
		\ifthenelse{\isempty{#1}}
		{\ensuremath{\textstyle\sum_{#2}}}
		{
			\ifthenelse{\equal{#1}{'}}
			{\ensuremath{\sideset{}{^{\prime}}{\textstyle\sum}_{#2}}}
			{\ensuremath{\sideset{}{^{#1}}{\textstyle\sum}_{#2}}}
		}
	}
}
\def\sumsstxt@ii[#1]#2[#3]{
	\@ifnextchar[
	{\sumsstxt@iii[#1]{#2}{#3}}
	{
		\ifthenelse{\isempty{#1}}
		{\ensuremath{\textstyle\sum_{\substack{#2 \\ #3}}}}
		{
			\ifthenelse{\equal{#1}{'}}
			{\ensuremath{\sideset{}{^{\prime}}{\textstyle\sum}_{\substack{#2 \\ #3}}}}
			{\ensuremath{\sideset{}{^{#1}}{\textstyle\sum}_{\substack{#2 \\ #3}}}}
		}
	}
}
\def\sumsstxt@iii[#1]#2#3[#4]{
	\@ifnextchar[
	{\sumsstxt@iv[#1]{#2}{#3}{#4}}
	{
		\ifthenelse{\isempty{#1}}
		{\ensuremath{\textstyle\sum_{\substack{#2 \\ #3 \\ #4}}}}
		{
			\ifthenelse{\equal{#1}{'}}
			{\ensuremath{\sideset{}{^{\prime}}{\textstyle\sum}_{\substack{#2 \\ #3 \\ #4}}}}
			{\ensuremath{\sideset{}{^{#1}}{\textstyle\sum}_{\substack{#2 \\ #3 \\ #4}}}}
		}
	}
}
\def\sumsstxt@iv[#1]#2#3#4[#5]{
	{\ifthenelse{\isempty{#1}}
		{\ensuremath{\textstyle\sum_{\substack{#2 \\ #3 \\ #4 \\ #5}}}}
		{
			\ifthenelse{\equal{#1}{'}}
			{\ensuremath{\sideset{}{^{\prime}}{\textstyle\sum}_{\substack{#2 \\ #3 \\ #4 \\ #5}}}}
			{\ensuremath{\sideset{}{^{#1}}{\textstyle\sum}_{\substack{#2 \\ #3 \\ #4 \\ #5}}}}
		}
	}
}

\def\prods{
	\@ifnextchar[
	{\prods@i}
	{\ensuremath{\prod}}
}
\def\prods@i[#1]{
	\@ifnextchar[
	{\prods@ii{#1}}
	{\ensuremath{\prod_{#1}}}
}
\def\prods@ii#1[#2]{
	\@ifnextchar[
	{\prods@iii{#1}{#2}}
	{\ensuremath{\prod_{\substack{#1 \\ #2}}}}
}
\def\prods@iii#1#2[#3]{
	\@ifnextchar[
	{\prods@iv{#1}{#2}{#3}}
	{\ensuremath{\prod_{\substack{#1 \\ #2 \\ #3}}}}
}
\def\prods@iv#1#2#3[#4]{
	{\ensuremath{\prod_{\substack{#1 \\ #2 \\ #3 \\ #4}}}}
}
\allowdisplaybreaks

\title[Piatetski--Shapiro primes in multiple Beatty sequences]
      {Piatetski--Shapiro primes in the intersection of multiple Beatty sequences}

\author[Victor Zhenyu Guo]{Victor Zhenyu Guo}

\address{School of Mathematics and Statistics, Xi'an Jiaotong University, Xi'an, Shaanxi, 710049, P. R. China}

\email{guozyv@xjtu.edu.cn}

\author[Jinjiang Li]{Jinjiang Li}

\address{Department of Mathematics, China University of Mining and Technology, Beijing, 100083, P. R. China}

\email{jinjiang.li.math@gmail.com}

\author[Min Zhang]{Min Zhang}

\address{School of Applied Science, Beijing Information Science and Technology University, Beijing, 100192, P. R. China}

\email{min.zhang.math@gmail.com}

\begin{document}
\footnotetext[1]{Jinjiang Li is the corresponding author.}

\begin{abstract}
Suppose that $\alpha_1, \alpha_2,\beta_1, \beta_2 \in\mathbb{R}$. Let $\alpha_1, \alpha_2 > 1$ be irrational and of finite type such that $1, \alpha_1^{-1}, \alpha_2^{-1}$ are linearly independent over $\mathbb{Q}$. Let $c$ be a real number in the range $1 < c < 12/11$. In this paper, it is proved that there exist infinitely many primes in the intersection of Beatty sequences $\cB_{\alpha_1,\beta_1} = \fl{\alpha_1 n + \beta_1}, \cB_{\alpha_2, \beta_2} = \fl{\alpha_2 n + \beta_2}$ and the Piatetski--Shapiro sequence $\sN^{(c)} = \fl{n^c}$. Moreover, we also give a sketch proof of Piatetski--Shapiro primes in the intersection of multiple Beatty sequences.
\end{abstract}

\maketitle

\begin{quote}
\end{quote}

\begin{quote}
\end{quote}



\section{Introduction}
\label{sec:intro}

The Piatetski--Shapiro sequences are sequences of the form
$$
\sN^{(c)} \defeq (\fl{n^c})_{n=1}^\infty\qquad(c>1,~c \not\in\NN).
$$
Such sequences have been named in honor of Piatetski--Shapiro, who \cite{PS}, in 1953, proved that $\sN^{(c)}$ contains infinitely many primes provided that $c\in(1,\frac{12}{11})$. More precisely, for such $c$ he showed that the counting function
$$
\pi^{(c)}(x) \defeq \# \Big\{\text{\rm prime~}p\le x : p\in\sN^{(c)}\Big\}
$$
satisfies the asymptotic property
$$
\pi^{(c)}(x) \sim \frac{x^{1/c}}{\log x} \qquad \text{~\rm as } x \to \infty.
$$
The range for $c$ of the above asymptotic formula in which it is known that $\sN^{(c)}$ contains infinitely many primes has been enlarged many times over the years and is currently known to hold for all $c\in(1,\frac{2817}{2426})$ thanks to Rivat and Sargos~\cite{RiSa}. Rivat and Wu~\cite{RiWu} also showed that there exist infinitely many Piatetski--Shapiro primes for $c \in (1, \frac{243}{205})$ by showing a lower bound of $\pi^{(c)}(x)$ with the expected order of magnitude. The same result is expected to hold for all larger values of $c$.  We remark that if $c\in(0,1)$ then $\sN^{(c)}$ contains all natural numbers, and hence all primes, particularly. The investigation of Piatetski--Shapiro primes is an approximation of the well--known conjecture that there exist infinitely many primes of the form $n^2+1$.

For fixed real numbers $\alpha$ and $\beta$, the associated non--homogeneous Beatty sequence is the sequence of integers defined by
$$
\cB_{\alpha,\beta} \defeq \(\fl{\alpha n+\beta}\)_{n=1}^\infty,
$$
where $\fl{t}$ denotes the integer part of any $t\in\RR$. Such sequences are also called generalized arithmetic progressions. If $\alpha$ is irrational, it follows from a classical exponential sum estimate of Vinogradov~\cite{Vino} that $\cB_{\alpha,\beta}$ contains infinitely many prime numbers; in fact, one has the asymptotic estimate
$$
\#\big\{\text{prime~}p\le x:p\in\cB_{\alpha,\beta}\big\}\sim
\alpha^{-1}\pi(x)  \qquad \text{\rm as\quad } x \to \infty,
$$
where $\pi(x)$ is the prime counting function. Moreover, Harman~\cite{Harm2} investigated the intersection of multiple Beatty sequences and showed the following result:

\emph{
\linebreak
Let $\xi$ be a positive integer and real numbers $\alpha_1, \dots, \alpha_{\xi}$ each exceeding $1$ be given such that
$$
1, \alpha_1^{-1}, \alpha_2^{-1}, \dots, \alpha_{\xi}^{-1} \text{ are linearly independent over } \mathbb{Q}.
$$
Then for any real numbers $\beta_1, \dots, \beta_{\xi}$, the intersection of $\cB_{\alpha_1,\beta_1}, \cB_{\alpha_2, \beta_2}, \dots, \cB_{\alpha_{\xi}, \beta_{\xi}}$ contains infinitely many primes. Indeed, the number of such primes up to $x$ equals to
$$
\frac{x}{\alpha_1 \cdots \alpha_{\xi} \log x} \big( 1 + o(1) \big).
$$
}

\noindent
It has been investigated in \cite{Guo1} that for $c\in(1,\frac{14}{13})$, there exist infinitely many primes in the intersection of a Beatty sequence and a Piatetski--Shapiro sequence with an asymptotic formula. Motivated by Harman's proof related to the intersection of multiple Beatty sequences, we show a corresponding result that there exist infinitely many primes in the intersection of multiple Beatty sequences and a Piatetski--Shapiro sequence. In this paper, we generalize and improve the previous theorem in~\cite{Guo1} and establish the following theorem.

\begin{theorem}
\label{thm:main}
Suppose that $\alpha_1, \alpha_2,\beta_1, \beta_2 \in\RR$. Let $\alpha_1, \alpha_2 > 1$ be irrational and of finite type such that
$$
1, \alpha_1^{-1}, \alpha_2^{-1} \text{ are linearly independent over } \mathbb{Q}.
$$
For $c \in (1, \frac{12}{11})$, there exist infinitely many primes in the intersection of Beatty sequences $\cB_{\alpha_1,\beta_1}, \cB_{\alpha_2, \beta_2}$ and the Piatetski--Shapiro sequence $\sN^{(c)}$. Moreover, the counting function
$$
\pi^{(c)}_{\alpha_1,\beta_1; \alpha_2, \beta_2} (x) \defeq \# \big\{ \text{\rm prime~}p \le x: p \in \cB_{\alpha_1,\beta_1} \cap \cB_{\alpha_2, \beta_2} \cap \sN^{(c)} \big\}
$$
satisfies
$$
\pi^{(c)}_{\alpha_1,\beta_1; \alpha_2, \beta_2} (x) = \frac{x^{1/c}}{{\alpha_1 \alpha_2} \log x} +O\(\frac{x^{1/c}}{\log^2x}\),
$$
where the implied constant depends only on $\alpha_1,\alpha_2$ and $c$.
\end{theorem}

By a similar method, we can also give a stronger revision of the theorem in \cite{Guo1}.

\begin{theorem}
\label{thm:2}
Suppose that $\alpha,\beta \in\RR$. Let $\alpha > 1$ be irrational and of finite type. For $c \in (1, \frac{12}{11})$, there exist infinitely many primes in the intersection of the Beatty sequence $\cB_{\alpha,\beta}$ and the Piatetski--Shapiro sequence $\sN^{(c)}$. Moreover, the counting function
$$
\pi^{(c)}_{\alpha,\beta} (x) \defeq \# \big\{ \text{\rm prime~}p \le x: p \in \cB_{\alpha,\beta} \cap \sN^{(c)} \big\}
$$
satisfies
$$
\pi^{(c)}_{\alpha,\beta} (x) = \frac{x^{1/c}}{{\alpha} \log x} +O\(\frac{x^{1/c}}{\log^2x}\),
$$
where the implied constant depends only on $\alpha$ and $c$.	
\end{theorem}	

In the end, we give the corresponding result to Harman's theorem in our case, which proves that there exist infinitely many primes in the intersection of a Piatetski--Shapiro sequence and multiple Beatty sequences.

\begin{theorem}
\label{thm:3}
 Suppose that $\xi$ is a positive integer, and $\alpha_1, \dots, \alpha_{\xi}, \beta_1, \dots, \beta_{\xi} \in\RR$. Let $\alpha_1, \dots, \alpha_{\xi} > 1$ be irrational and of finite type such that
$$
1, \alpha_1^{-1}, \dots, \alpha_{\xi}^{-1} \text{ are linearly independent over } \mathbb{Q}.
$$
For $c \in (1, \frac{12}{11})$, there are infinitely many primes in the intersection of Beatty sequences $\cB_{\alpha_1,\beta_1}, \dots, \cB_{\alpha_{\xi}, \beta_{\xi}}$ and the Piatetski--Shapiro sequence $\sN^{(c)}$.
Moreover, the counting function
$$
\pi^{(c)}_{\alpha_1,\beta_1; \dots; \alpha_{\xi},\beta_{\xi}} (x) \defeq \# \big\{ \text{\rm prime~}p \le x: p \in \cB_{\alpha_1,\beta_1} \cap \cdots \cap \cB_{\alpha_{\xi}, \beta_{\xi}} \cap \sN^{(c)} \big\}
$$
satisfies
$$
\pi^{(c)}_{\alpha_1,\beta_1; \dots; \alpha_{\xi},\beta_{\xi}} (x) = \frac{x^{1/c}}{{\alpha_1 \cdots \alpha_{\xi}} \log x} +O\(\frac{x^{1/c}}{\log^2x}\),
$$
where the implied constant depends only on $\alpha_1, \dots, \alpha_{\xi}$ and $c$.
\end{theorem}

\section{Preliminaries}

\subsection{Notation}

We denote by $\fl{t}$ and $\{t\}$ the integral part and the fractional part of $t$, respectively.
As is customary, we put
$$
\e(t)\defeq e^{2\pi it}\mand
\{ t \} \defeq t-\fl{t}.
$$
Throughout the paper, we make considerable use of the sawtooth function defined by
$$
\psi(t) \defeq t-\fl{t}-\frac{1}{2} = \{t\}-\frac{1}{2}.
$$
The notation $\| t \|$ is used to denote the distance
from the real number $t$ to the nearest integer; that is,
$$
\| t\| \defeq\min_{n\in\ZZ}|t-n|.
$$

Let $\PP$ denote the set of primes in $\NN$. The letter $p$ always denotes a prime. For a Beatty sequence $(\fl{\alpha n + \beta})_{n=1}^\infty$, we denote $\omega \defeq \alpha^{-1}$. We represent $\gamma \defeq c^{-1}$ for the Piatetski--Shapiro sequence $(\fl{n^c})_{n=1}^\infty$. We use notation of the form $m\sim M$ as an abbreviation for $M< m\le 2M$.

Throughout the paper, $\varepsilon$ always denotes an arbitrarily small positive constant, which may not be the
same at different occurrences; the implied constants in symbols $O$, $\ll$ and $\gg$ may depend (where obvious) on the parameters $\alpha_1, \dots, \alpha_{\xi}, \beta_1, \dots, \beta_{\xi}, c, \eps$ but are absolute otherwise. For given
functions $F$ and $G$, the notations $F\ll G$, $G\gg F$ and $F=O(G)$ are all equivalent to the statement that the inequality
$|F|\le \mathcal{C}|G|$ holds with some constant $\mathcal{C}>0$.

\subsection{Type of an irrational number}

For any irrational number $\alpha$, we define its type $\tau=\tau(\alpha)$ by the following definition
$$
\tau\defeq\sup\Big\{t\in\RR:\liminf\limits_{n\to\infty}n^t\| \alpha n \|=0\Big\}.
$$
Using Dirichlet's approximation theorem, one can see that $\tau\ge 1$ for every irrational number $\alpha$. Thanks to
the work of Khinchin \cite{Khin} and Roth \cite{Roth1,Roth2}, it is known that $\tau=1$ for almost all real numbers,
in the sense of the Lebesgue measure, and for all irrational algebraic numbers, respectively. Moreover, if $\alpha$ is an irrational number of type $\tau<\infty$, then so are $\alpha^{-1}$ and $n\alpha^{-1}$ for all integer $n\ge 1$.

\subsection{Technical lemmas}

We need the following well--known approximation of Vaaler \cite{Vaal}.

\begin{lemma}
\label{lem:Vaaler}
For any $H\ge 1$, there exist numbers $a_h,b_h$ such that
$$
\bigg|\psi(t)-\sum_{0<|h|\le H}a_h\,\e(th)\bigg|
\le\sum_{|h|\le H}b_h\,\e(th),\qquad
a_h\ll\frac{1}{|h|}\,,\qquad b_h\ll\frac{1}{H}\,.
$$
\end{lemma}

\begin{lemma}
\label{lem:index}
For an arithmetic function $g$ and $N'\sim N$, we have
$$
\sum_{N < p\le N'} g(p) \ll \frac{1}{\log N} \max_{N<N_1 \le 2N}
\bigg|\sum_{N < n \le N_1} \Lambda(n)g(n)\bigg| + N^{1/2}.
$$
\end{lemma}
\begin{proof}
See the argument on page 48 of \cite{GraKol}.
\end{proof}

\begin{lemma}\label{lem:amn}
Suppose that
\begin{equation*}
\alpha=\frac{a}{q}+\frac{\theta}{q^2},
\end{equation*}
with $(a,q)=1, q\geqslant1, |\theta|\leqslant 1$. Then there holds
\begin{equation*}
\sum_{m\leqslant N}\Lambda(m)\mathbf{e}(m\alpha)\ll\big(Nq^{-1/2}+N^{4/5}+N^{1/2}q^{1/2}\big)(\log N)^4.
\end{equation*}
\end{lemma}
\begin{proof}
See Chapter 25 of Davenport \cite{Daven}.
\end{proof}

\begin{lemma}\label{lem:ttr}
Suppose that $a$ is a fixed irrational number of finite type $\tau < \infty$ and $h \geqslant1,m$ are integers. Then we have
\begin{equation*}
\sum_{m\leqslant M}\Lambda(m)\mathbf{e}(ahm) \ll h^{1/2} M^{1-1/(2\tau) + \varepsilon}+M^{1-\varepsilon}.
\end{equation*}
\end{lemma}

\begin{proof}
For any sufficiently small $\varepsilon>0$, we set $\varrho=\tau+\varepsilon$. Since $a$ is of type $\tau$, there exists some constant $\mathfrak{c}>0$ such that
\begin{equation}\label{eq:dis}
\|an\|>\mathfrak{c}n^{-\varrho},\qquad n\geqslant 1.
\end{equation}
For given $h$ with $0<h\leqslant H$, let $b/d$ be the convergent in the continued fraction expansion of $ah$, which has the largest denominator $d$ not exceeding $M^{1-\eta}$ for a sufficiently small positive number $\eta$. Then we derive that
\begin{equation}\label{eq:apr}
\bigg|ah-\frac{b}{d}\bigg|\leqslant\frac{1}{dM^{1-\eta}}\leqslant\frac{1}{d^2},
\end{equation}
which combined with \eqref{eq:dis} yields
$$
M^{-1+\eta}\geqslant|ahd-b|\geqslant\|ahd\|>\mathfrak{c}(hd)^{-\varrho}.
$$
Taking $C_0:=\mathfrak{c}^{1/\varrho}$, we obtain
\begin{equation}\label{eq:d}
d > C_0h^{-1}M^{1/\varrho - \eta/\varrho}.
\end{equation}
Combining~\eqref{eq:apr} and~\eqref{eq:d}, applying Lemma~\ref{lem:amn} and the fact that $d \le M^{1-\eta}$, we deduce that
\begin{align*}
           \sum_{m \leqslant M}\Lambda(m)\mathbf{e}(ahm)
 \ll &\,\, \big( M d^{-1/2} + M^{4/5} + M^{1/2} d^{1/2} \big) (\log M)^4 \\
 \ll &\,\, \big( h^{1/2} M^{1-1/(2\varrho)+\eta/(2\varrho)}+ M^{4/5} +  M^{1-\eta/2} \big) (\log M)^4 \\
 \ll &\,\, h^{1/2} M^{1-1/(2\tau) + \varepsilon}+ M^{1-\varepsilon}.
\end{align*}
This completes the proof of Lemma \ref{lem:ttr}.
\end{proof}

\begin{lemma}\label{lem:finite-type}
Suppose that $\xi$ is an integer, and $\omega_1,\omega_2,\dots,\omega_\xi$ are irrational numbers of finite type such that
\begin{equation*}
   1,\omega_1,\omega_2,\dots,\omega_\xi
\end{equation*}
are linearly independent over $\mathbb{Q}$. Then for any subset  $\{i_1,i_2,\dots,i_s\}\subset\{1,2,\dots,\xi\}$, the
combination of $h_{i_1}\omega_{i_1}+\cdots+h_{i_s}\omega_{i_s}$ with $h_{i_1},\dots,h_{i_s}\in\mathbb{N}^\times$ is of finite type.
\end{lemma}
\begin{proof}
For any $n\geqslant1$, we have
\begin{align*}
\big\|n(h_{i_1}\omega_{i_1}+\cdots+h_{i_s}\omega_{i_s})\big\|\leqslant
\big\|nh_{i_1}\omega_{i_1}\big\|+\cdots+\big\|nh_{i_s}\omega_{i_s}\big\|.
\end{align*}	
By the definition of the type of irrational number, we deduce that
\begin{align*}
 \tau(h_{i_1}\omega_{i_1}+\cdots+h_{i_s}\omega_{i_s})\leqslant \tau(h_{i_1}\omega_{i_1})+\cdots+\tau(h_{i_s}\omega_{i_s}),
\end{align*}
which implies the finite type of $h_{i_1}\omega_{i_1}+\cdots+h_{i_s}\omega_{i_s}$.
\end{proof}

The following lemma gives a characterization of the numbers in the Beatty sequence $\cB_{\alpha, \beta}$.
\begin{lemma}\label{lem:Beatty}
A natural number $m$ has the form $\fl{\alpha n + \beta}$ if and only if $\cX_{\alpha, \beta} (m) = 1$, where $\cX_{\alpha, \beta} (m) \defeq \fl{-\alpha^{-1} (m-\beta)} - \fl{-\alpha^{-1} (m+1-\beta)}$.	
\end{lemma}
\begin{proof}
Note that an integer $m$ has the form $m= \fl{\alpha n+\beta} $ for some integer $n$ if and only if
\begin{equation*}
\frac{m-\beta}{\alpha}\leqslant n<\frac{m-\beta+1}{\alpha}.
\end{equation*}	
\end{proof}

Finally, we use the following lemma, which provides a characterization of the numbers that
occur in the Piatetski--Shapiro sequence $\sN^{(c)}$.

\begin{lemma}
\label{lem:PS}
A natural number $m$ has the form $\fl{n^c}$ if and only if $\cX^{(c)}(m) = 1$, where
$\cX^{(c)}(m) \defeq \fl{-m^\gamma} - \fl{-(m+1)^\gamma}$.  Moreover,
$$
\cX^{(c)}(m)=\gamma m^{\gamma-1} + \psi(-(m+1)^\gamma) - \psi(-m^\gamma) + O(m^{\gamma-2}).
$$
\end{lemma}
\begin{proof}
The proof of Lemma \ref{lem:PS} is similar to that of Lemma \ref{lem:Beatty}, so we omit the details herein.	
\end{proof}

\begin{lemma}
For $1<c<\frac{2817}{2426}$, there holds
\begin{equation} \label{eq:PSthm}
\pi^{(c)}(x) = \sum_{p \le x} \cX^{(c)}(p) = \frac{x^\gamma}{\log x}
+ O\bigg(\frac{x^\gamma}{\log^2 x}\bigg).
\end{equation}
\end{lemma}
\begin{proof}
See Theorem 1 of Rivat and Sargos \cite{RiSa}.	
\end{proof}

\subsection{Upper bound estimate of exponential sums}
We begin with the decomposition of the von Mangoldt function by Heath--Brown.

\begin{lemma}\label{Heath-Brown-identity}
	Let $z\geqslant1$ and $k\geqslant1$. Then, for any $n\leqslant2z^k$, there holds
	\begin{equation*}
		\Lambda(n)=\sum_{j=1}^k(-1)^{j-1}\binom{k}{j}\mathop{\sum\cdots\sum}_{\substack{n_1n_2\cdots n_{2j}=n\\
				n_{j+1},\dots,n_{2j}\leqslant z }}(\log n_1)\mu(n_{j+1})\cdots\mu(n_{2j}).
	\end{equation*}
\end{lemma}
\begin{proof}
	See the arguments on pp. 1366--1367 of Heath--Brown \cite{HB}.
\end{proof}

\begin{lemma}\label{compare}
	Suppose that
	\begin{equation*}
		L(H)=\sum_{i=1}^mA_iH^{a_i}+\sum_{j=1}^nB_jH^{-b_j},
	\end{equation*}
	where~$A_i,\,B_j,\,a_i$~and~$b_j$~are positive. Assume further that $H_1\leqslant H_2$. Then there exists
	some $\mathscr{H}$ with
	$H_1\leqslant\mathscr{H}\leqslant H_2$ and
	\begin{equation*}
		L(\mathscr{H})\ll \sum_{i=1}^mA_iH_1^{a_i}+\sum_{j=1}^nB_jH_2^{-b_j}+\sum_{i=1}^m\sum_{j=1}^n\big(A_i^{b_j}B_j^{a_i}\big)^{1/(a_i+b_j)}.
	\end{equation*}
	The implied constant depends only on $m$ and $n$.
\end{lemma}
\begin{proof}
	See Lemma 3 of Srinivasan \cite{Srin}.
\end{proof}

For real numbers $m_1$ and $m_2$, the sum of the form
\begin{equation*}
	\mathop{\sum_{k\sim K}\sum_{\ell\sim L}}_{KL\asymp x}a_kb_\ell \mathbf{e} \big( h(k\ell)^\gamma + m_1k \ell + m_2 \big)
\end{equation*}
with $|a_k|\ll x^\varepsilon,|b_\ell|\ll x^\varepsilon$ for every fixed $\varepsilon>0$, it is usually called a ``Type I'' sum, denoted by $S_I(K,L)$, if $b_\ell=1$ or $b_\ell=\log\ell$; otherwise it is called a ``Type II'' sum, denoted by $S_{II}(K,L)$.

\begin{lemma}\label{derivative-estimate}
	Suppose that $f(x):[a, b]\to \mathbb{R}$ has continuous derivatives of arbitrary
	order on $[a,b]$, where $1\leqslant a<b\leqslant2a$. Suppose further that
	\begin{equation*}
		\big|f^{(j)}(x)\big|\asymp \lambda_j,\qquad j\geqslant1, \qquad x\in[a, b].
	\end{equation*}
	Then we have
	\begin{equation}\label{2nd-deri-estimate}
		\sum_{a<n\leqslant b}e\big(f(n)\big)\ll a\lambda_2^{1/2}+\lambda_2^{-1/2},
	\end{equation}
	and
	\begin{equation}\label{3rd-deri-estimate}
		\sum_{a<n\leqslant b}e\big(f(n)\big)\ll a\lambda_3^{1/6}+\lambda_3^{-1/3}.
	\end{equation}
\end{lemma}
\begin{proof}
	For (\ref{2nd-deri-estimate}), one can see Corollary 8.13 of Iwaniec and Kowalski \cite{IwKo}, or Theorem 5
	of Chapter 1 in Karatsuba \cite{Kara}. For (\ref{3rd-deri-estimate}), one can see Corollary 4.2 of
	Sargos \cite{Sarg}.
\end{proof}

\begin{lemma}\label{Type-I-sum}
Suppose that $|a_k|\ll 1,b_\ell=1$ or $\log\ell,KL\asymp x$. Then if $K\ll x^{1/2}$, there holds
\begin{equation*}
S_{I}(K,L)\ll |h|^{1/6}x^{\gamma/6+3/4 + \eps} + |h|^{-1/3}x^{1-\gamma/3 + \eps}.
\end{equation*}
\end{lemma}
\begin{proof}
Set $f(\ell)=h(k\ell)^\gamma+m_1 k\ell + m_2$. It is easy to see that
\begin{equation*}
f'''(\ell)=\gamma(\gamma-1)(\gamma-2)hk^\gamma\ell^{\gamma-3}\asymp |h| K^{\gamma}L^{\gamma-3}.
\end{equation*}
If $K\ll x^{1/2}$, then by (\ref{3rd-deri-estimate}) of Lemma \ref{derivative-estimate}, we deduce that
\begin{align*}
x^{-\eps} \cdot S_{I}(K,L)
\ll & \,\, \sum_{k\sim K} \Bigg|\sum_{\ell\sim L}\mathbf{e}\big(f(\ell)\big)\Bigg|
               \nonumber \\
\ll & \,\, \sum_{k\sim K} \Big( L \big( |h|K^{\gamma}L^{\gamma-3} \big)^{1/6} +
           \big( |h|K^{\gamma}L^{\gamma-3} \big)^{-1/3} \Big)
               \nonumber \\
\ll & \,\, |h|^{1/6}x^{\gamma/6+1/2}K^{1/2} + |h|^{-1/3}x^{1-\gamma/3}
               \nonumber \\
\ll & \,\, |h|^{1/6}x^{\gamma/6+3/4}+|h|^{-1/3}x^{1-\gamma/3},
\end{align*}
which completes the proof of Lemma \ref{Type-I-sum}.
\end{proof}

\begin{lemma}\label{Type-II-sum}
Suppose that $|a_k|\ll 1,|b_\ell|\ll1$ with $k\sim K,\ell\sim L$ and $KL\asymp x$. Then if $x^{1/2}\ll K\ll x^{19/25}$, there holds
\begin{equation*}
S_{II}(K,L) \ll |h|^{1/4} x^{\gamma/4+5/8} + |h|^{-1/4}x^{1-\gamma/4} + x^{22/25} + |h|^{1/6}x^{\gamma/6+3/4}.
	\end{equation*}
\end{lemma}
\begin{proof}
Let $Q$, which satisfies $1 < Q < L$, be a parameter which will be chosen later. By the Weyl--van der Corput inequality
(see Lemma 2.5 of Graham and Kolesnik \cite{GraKol}), we have
\begin{align}\label{Weyl-inequality}
&\,\, \Bigg| \mathop{\sum_{k\sim K} \sum_{\ell\sim L}}_{KL\asymp x} a_k b_\ell \mathbf{e} \big(h(k\ell)^\gamma + m_1 k\ell + m_2\big)\Bigg|^2 \nonumber \\
\ll &\,\, K^2 L^2 Q^{-1} + KLQ^{-1} \sum_{\ell\sim L} \sum_{0<|q|\leqslant Q} \big|\mathfrak{S}(q;\ell)\big|,
\end{align}
where
\begin{equation*}
\mathfrak{S}(q;\ell)=\sum_{k\in\mathcal{I}(q;\ell)}\mathbf{e}\big(g(k)\big)
\end{equation*}
with
\begin{equation*}
g(k)=hk^\gamma\big(\ell^\gamma-(\ell+q)^\gamma\big)- m_1 kq.
\end{equation*}
It is easy to see that
\begin{equation*}
g''(k)=\gamma(\gamma-1) hk^{\gamma-2} \big(\ell^\gamma-(\ell+q)^\gamma\big) \asymp |h| K^{\gamma-2} L^{\gamma-1} |q|.
\end{equation*}
By (\ref{2nd-deri-estimate}) of Lemma \ref{derivative-estimate}, we have
\begin{equation}\label{Type-II-inner}
\mathfrak{S}(q;\ell)\ll K\big(|h| K^{\gamma-2} L^{\gamma-1} |q| \big)^{1/2} + \big(|h| K^{\gamma-2}L^{\gamma-1}|q|\big)^{-1/2}.
\end{equation}
Putting (\ref{Type-II-inner}) into (\ref{Weyl-inequality}), we derive that
\begin{align*}
& \,\, \Bigg|\mathop{\sum_{k\sim K} \sum_{\ell\sim L}}_{KL\asymp x} a_k b_\ell \mathbf{e} \big( h (k\ell)^\gamma + m_1 k \ell + m_2 \big)\Bigg|^2
\nonumber \\
\ll & \,\, K^2L^2Q^{-1}+KLQ^{-1}
\nonumber \\
& \,\,\times\sum_{\ell\sim L}\sum_{0<|q|\leqslant Q}
\big(|h|^{1/2}K^{\gamma/2}L^{\gamma/2-1/2}|q|^{1/2}+|h|^{-1/2}K^{1-\gamma/2}L^{1/2-\gamma/2}|q|^{-1/2}\big)
\nonumber \\
\ll & \,\,K^2L^2Q^{-1}+KLQ^{-1}\big(|h|^{1/2}K^{\gamma/2}L^{\gamma/2+1/2}Q^{3/2} \\
&\,\, + |h|^{-1/2}K^{1-\gamma/2}L^{3/2-\gamma/2}Q^{1/2}\big)
\nonumber \\
\ll & \,\, K^2L^2Q^{-1}+|h|^{1/2}K^{1+\gamma/2}L^{\gamma/2+3/2}Q^{1/2}+|h|^{-1/2}K^{2-\gamma/2}L^{5/2-\gamma/2}Q^{-1/2}.
\end{align*}
By noting that $1\leqslant Q\leqslant L$, it follows from Lemma \ref{compare} that there exists an optimal $Q$ such that
\begin{align*}
& \,\, \Bigg|\mathop{\sum_{k\sim K}\sum_{\ell\sim L}}_{KL\asymp x}a_kb_\ell \mathbf{e}\big(h (k\ell)^\gamma + m_1 k \ell + m_2\big)\Bigg|^2
\nonumber \\
\ll & \,\, |h|^{1/2}x^{\gamma/2+3/2}K^{-1/2}+Kx+|h|^{-1/2}x^{2-\gamma/2}+|h|^{1/3}x^{\gamma/3+5/3}K^{-1/3}+K^{-1/2}x^2,
\end{align*}
which implies
\begin{align*}
&\,\, \Bigg|\mathop{\sum_{k\sim K}\sum_{\ell\sim L}}_{KL\asymp x}a_kb_\ell \mathbf{e}\big(h (k\ell)^\gamma + m_1 k \ell + m_2\big)\Bigg| \\
\ll & \,\, |h|^{1/4}x^{\gamma/4+3/4}K^{-1/4}+|h|^{-1/4}x^{1-\gamma/4} + K^{1/2}x^{1/2} \\
&\,\, + |h|^{1/6}x^{\gamma/6+5/6}K^{-1/6} + K^{-1/4}x \\
\ll &\,\, |h|^{1/4} x^{\gamma/4+5/8} + |h|^{-1/4}x^{1-\gamma/4} + x^{22/25} + |h|^{1/6}x^{\gamma/6+3/4},
\end{align*}
provided that $x^{1/2}\ll K\ll x^{19/25}$, which completes the proof of Lemma \ref{Type-II-sum}.
\end{proof}

\begin{lemma}\label{lem:exp}
For real numbers $m_1, m_2$ and $x/2<v\le x$, we have
\begin{align*}
&\,\,\max_{x/2<v\leqslant x} \Bigg| \sum_{x/2< n \leqslant v} \Lambda(n) \mathbf{e}(hn^\gamma+m_1n+m_2)\Bigg| \\
\ll &\,\, x^{\eps} \Big( |h|^{1/6}x^{\gamma/6+3/4} + |h|^{-1/3}x^{1-\gamma/3} + |h|^{1/4} x^{\gamma/4+5/8} \\
&\,\, \qquad + |h|^{-1/4}x^{1-\gamma/4} + x^{22/25} \Big).
\end{align*}	
\end{lemma}
\begin{proof}
By Heath--Brown's identity, i.e. Lemma \ref{Heath-Brown-identity}, with $k=3$, one can see that the exponential sum
\begin{equation*}
\max_{x/2<t\leqslant x} \Bigg|\sum_{x/2 < n \leqslant t} \Lambda(n) \mathbf{e}(hn^\gamma + m_1 n + m_2)\Bigg|
\end{equation*}
can be written as linear combination of $O(\log^6x)$ sums, each of which is of the form
\begin{align}\label{single-sum}
\mathcal{T}^* \defeq & \,\, \sum_{n_1\sim N_1}\cdots \sum_{n_6\sim N_6}(\log n_1)\mu(n_4)\mu(n_5)\mu(n_6) \nonumber \\
& \,\, \qquad\qquad\qquad \times \mathbf{e}\big(h(n_1n_2\cdots n_6)^\gamma + (n_1n_2\cdots n_6)m_1 + m_2\big),
\end{align}
where $N_1N_2\cdots N_6\asymp x$; $2N_i\leqslant(2x)^{1/3},i=4,5,6$ and some $n_i$ may only take value $1$.
Therefore, it is sufficient for us to give upper bound estimate for each $\mathcal{T}^*$ defined as in (\ref{single-sum}). Next, we will consider three cases.

\noindent
\textbf{Case 1.} If there exists an $N_j$ such that $N_j\geqslant x^{1/2}$, then we must have $j\leqslant3$ for the fact that
$N_i\ll x^{1/3}$ with $i=4,5,6$. Let
\begin{equation*}
	k=\prod_{\substack{1\leqslant i\leqslant6\\ i\not=j}}n_i,\qquad \ell=n_j,\qquad
	K=\prod_{\substack{1\leqslant i\leqslant6\\ i\not=j}}N_i,\qquad L=N_j.
\end{equation*}
In this case, we can see that $\mathcal{T}^*$ is a sum of ``Type I'' satisfying $K\ll x^{1/2}$. By Lemma \ref{Type-I-sum},
we have
\begin{equation*}
x^{-\varepsilon}\cdot \mathcal{T}^*\ll |h|^{1/6}x^{\gamma/6+3/4} + |h|^{-1/3}x^{1-\gamma/3}.
\end{equation*}

\noindent
\textbf{Case 2.} If there exists an $N_j$ such that $x^{6/25}\leqslant N_j<x^{1/2}$, then we take
\begin{equation*}
	k=\prod_{\substack{1\leqslant i\leqslant6\\ i\not=j}}n_i,\qquad \ell=n_j,\qquad
	K=\prod_{\substack{1\leqslant i\leqslant6\\ i\not=j}}N_i,\qquad L=N_j.
\end{equation*}
Thus, $\mathcal{T}^*$ is a sum of ``Type II'' satisfying $x^{1/2}\ll K\ll x^{19/25}$. By Lemma \ref{Type-II-sum},
we have
\begin{equation*}
	x^{-\varepsilon}\cdot  \mathcal{T}^*\ll |h|^{1/4} x^{\gamma/4+5/8} + |h|^{-1/4}x^{1-\gamma/4} + x^{22/25} + |h|^{1/6}x^{\gamma/6+3/4}.
\end{equation*}

\noindent
\textbf{Case 3.} If $N_j<x^{6/25}\,(j=1,2,3,4,5,6)$, without loss of generality, we assume that
$N_1\geqslant N_2\geqslant\cdots\geqslant N_6$. Let $r$ denote the natural number $j$ such that
\begin{equation*}
	N_1N_2\cdots N_{j-1}<x^{6/25},\qquad N_1N_2\cdots N_j\geqslant x^{6/25}.
\end{equation*}
Since $N_1<x^{6/25}$ and $N_6<x^{6/25}$, then $2\leqslant r\leqslant5$. Thus, we have
\begin{equation*}
	x^{6/25}\leqslant N_1N_2\cdots N_r=(N_1\cdots N_{r-1})\cdot N_r<x^{6/25}\cdot x^{6/25}<x^{1/2}.
\end{equation*}
Let
\begin{equation*}
	k=\prod_{i=r+1}^6n_i,\qquad \ell=\prod_{i=1}^rn_i,\qquad K=\prod_{i=r+1}^6N_i,\qquad L=\prod_{i=1}^rN_i.
\end{equation*}
At this time, $\mathcal{T}^*$ is a sum of ``Type II'' satisfying $x^{1/2}\ll K\ll x^{19/25}$. By Lemma \ref{Type-II-sum},
we have
\begin{equation*}
	x^{-\varepsilon}\cdot \mathcal{T}^*\ll |h|^{1/4} x^{\gamma/4+5/8} + |h|^{-1/4}x^{1-\gamma/4} + x^{22/25} + |h|^{1/6}x^{\gamma/6+3/4}.
\end{equation*}
Combining the above three cases, we derive that
\begin{align*}
	x^{-\varepsilon}\cdot \mathcal{T}^*
	\ll & \,\, |h|^{1/6}x^{\gamma/6+3/4} + |h|^{-1/3}x^{1-\gamma/3} \\
	&\,\, + |h|^{1/4} x^{\gamma/4+5/8} + |h|^{-1/4}x^{1-\gamma/4} + x^{22/25},
\end{align*}
which completes the proof of this lemma.
\end{proof}

\section{Proof of Theorem~\ref{thm:main}}

\subsection{Initial construction and the main term}

We start by two Beatty sequences
$$
\cB_{\alpha_1, \beta_1} \defeq \fl{\alpha_1 n + \beta_1} \mand \cB_{\alpha_2, \beta_2} \defeq \fl{\alpha_2 n + \beta_2}.
$$
We define that $\omega_1 \defeq \alpha_1^{-1}$ and $\omega_2 \defeq \alpha_2^{-1}$. By the definition of $\pi^{(c)}_{\alpha_1,\beta_1; \alpha_2, \beta_2} (x)$, we obtain that
\begin{equation}\label{char-represent}
\pi^{(c)}_{\alpha_1,\beta_1; \alpha_2, \beta_2} (x) = \sum_{p \le x} \cX_{\alpha_1, \beta_1} (p) \cX_{\alpha_2, \beta_2} (p) \cX^{(c)}(p),
\end{equation}
where by Lemma~\ref{lem:Beatty}
$$
\cX_{\alpha_i, \beta_i} (p) \defeq \fl{-\omega_i(p-\beta_i)} - \fl{-\omega_i(p+1-\beta_i)}, \quad (i=1,2),
$$
and by Lemma~\ref{lem:PS}
$$
\cX^{(c)}(p) \defeq \fl{-p^\gamma} - \fl{-(p+1)^\gamma}.
$$
Moreover, we see that
\begin{equation}\label{represent-1}
\cX_{\alpha_i, \beta_i} (p) = \omega_i + \psi(-\omega_i(p+1-\beta_i))-\psi(-\omega_i(p-\beta_i)),\quad (i=1,2),
\end{equation}
and
\begin{equation}\label{represent-2}
\cX^{(c)}(p) = \gamma p^{\gamma-1}+O(p^{\gamma-2})+\psi(-(p+1)^\gamma)-\psi(-p^\gamma).
\end{equation}
Combining (\ref{char-represent}), (\ref{represent-1}) and (\ref{represent-2}), we obtain
$$
\pi^{(c)}_{\alpha_1,\beta_1; \alpha_2, \beta_2} (x) = \cS_1 + \cS_2 + \cS_3 + \cS_4 + \cS_5 + \cS_6 + \cS_7,
$$
where
\begin{align*}
\cS_1 &\defeq \sum_{p \le x} \omega_1 \omega_2 \cX^{(c)}(p), \\
\cS_2 &\defeq \sum_{p \le x} \omega_1 \big(  \gamma p^{\gamma-1}+O(p^{\gamma-2}) \big) \\
&\qquad \times \big( \psi(-\omega_2(p+1-\beta_2))-\psi(-\omega_2(p-\beta_2)) \big), \\
\cS_3 &\defeq \sum_{p \le x} \omega_2 \big(  \gamma p^{\gamma-1}+O(p^{\gamma-2}) \big) \\
&\qquad \times \big( \psi(-\omega_1(p+1-\beta_1))-\psi(-\omega_1(p-\beta_1)) \big),	\\
\cS_4 &\defeq \sum_{p \le x} \big(  \gamma p^{\gamma-1}+O(p^{\gamma-2}) \big) \big( \psi(-\omega_2(p+1-\beta_2))-\psi(-\omega_2(p-\beta_2)) \big) \\
&\qquad \times \big( \psi(-\omega_1(p+1-\beta_1))-\psi(-\omega_1(p-\beta_1)) \big), \\
\cS_5 &\defeq \sum_{p \le x} \omega_1 \big( \psi(-(p+1)^\gamma)-\psi(-p^\gamma) \big) \\
&\qquad \times \big( \psi(-\omega_2(p+1-\beta_2))-\psi(-\omega_2(p-\beta_2)) \big), \\
\cS_6 &\defeq \sum_{p \le x} \omega_2 \big(  \psi(-(p+1)^\gamma)-\psi(-p^\gamma) \big) \\
&\qquad \times \big( \psi(-\omega_1(p+1-\beta_1))-\psi(-\omega_1(p-\beta_1)) \big),	\\
\cS_7 \displaystyle &\defeq \sum_{p \le x} \big( \psi(-(p+1)^\gamma)-\psi(-p^\gamma) \big)
\big( \psi(-\omega_2(p+1-\beta_2))-\psi(-\omega_2(p-\beta_2)) \big) \\
&\qquad \times \big( \psi(-\omega_1(p+1-\beta_1))-\psi(-\omega_1(p-\beta_1)) \big).
\end{align*}
By~\eqref{eq:PSthm}, we derive an asymptotic formula for $\cS_1$ with $c\in(1,\frac{2817}{2426})$, which is
$$
\cS_1 = \frac{x^\gamma}{\alpha_1 \alpha_2 \log x}+O\bigg(\frac{x^\gamma}{\log^2 x}\bigg).
$$
Next, for $i = 2, 3, 4, 5, 6, 7$, we shall prove that
$$
\cS_i \ll x^\gamma \log^{-2} x.
$$
Applying the Vaaler's approximation, i.e. Lemma~\ref{lem:Vaaler}, and taking
$$
H_1 = H_2 \defeq x^{\eps} \mand H_3 \defeq x^{1-\gamma+\eps}
$$
with a sufficiently small positive number $\eps$, we have that
\begin{align}\label{eq:eta1}
&\qquad \psi(-\omega_1(p+1-\beta_1))-\psi(-\omega_1(p-\beta_1)) \nonumber \\
&= \sum_{0 < |h_1| \le H_1} a_{h_1} \big( \e(\omega_1 h_1 (p+1-\beta_1)) - \e(\omega_1 h_1 (p - \beta_1)) \big) \nonumber \\
&\qquad + O\( \sum_{|h_1| \le H_1} b_{h_1} \big( \e(\omega_1 h_1 (p+1-\beta_1)) + \e(\omega_1 h_1 (p - \beta_1)) \big) \),
\end{align}
and
\begin{align}\label{eq:eta2}
&\qquad \psi(-\omega_2(p+1-\beta_2))-\psi(-\omega_2(p-\beta_2)) \nonumber \\
&= \sum_{0 < |h_2| \le H_2} a_{h_2} \big( \e(\omega_2 h_2 (p+1-\beta_2)) - \e(\omega_2 h_2 (p - \beta_2)) \big) \nonumber \\
&\qquad + O\( \sum_{|h_2| \le H_2} b_{h_2} \big( \e(\omega_2 h_2 (p+1-\beta_2)) + \e(\omega_2 h_2 (p - \beta_2)) \big) \),
\end{align}
and
\begin{align}\label{eq:gamma}
&\qquad \psi(-(p+1)^\gamma)-\psi(-p^\gamma) \nonumber \\
&= \sum_{0 < |h_3| \le H_3} a_{h_3} \big( \e(h_3 (p+1)^\gamma) - \e(h_3 p^\gamma ) \big) \nonumber \\
&\qquad + O\( \sum_{|h_3| \le H_3} b_{h_3} \big( \e(h_3 (p+1)^\gamma) + \e(h_3 p^\gamma ) \big) \).
\end{align}
We mention that for $j=1,2,3$, there holds
$$
a_{h_j} \ll |h_j|^{-1} \mand b_{h_j} \ll H_j^{-1}.
$$

\subsection{Upper bounds of $\cS_2$ and $\cS_3$}
In order to prove that $\cS_2 \ll x^\gamma \log^{-2} x$, we write $\cS_2 = \cS_{21} + O(\cS_{22})$, where
$$
\cS_{21} \defeq \sum_{p \le x} \omega_1 \gamma p^{\gamma-1}  \big( \psi(-\omega_2(p+1-\beta_2))-\psi(-\omega_2(p-\beta_2)) \big)
$$
and
$$
\cS_{22} \defeq \sum_{p \le x} \omega_1 p^{\gamma-2}  \big( \psi(-\omega_2(p+1-\beta_2))-\psi(-\omega_2(p-\beta_2)) \big).
$$
By~\eqref{eq:eta2}, we obtain that $\cS_{21} = \cS_{23} + O(\cS_{24})$, where
$$
\cS_{23} \defeq \sum_{p \le x} \omega_1 \gamma p^{\gamma-1} \sum_{0 < |h_2| \le H_2} a_{h_2} \big( \e(\omega_2 h_2 (p+1-\beta_2)) - \e(\omega_2 h_2 (p - \beta_2)) \big)
$$
and
$$
\cS_{24} \defeq \sum_{p \le x} \omega_1 \gamma p^{\gamma-1} \sum_{|h_2| \le H_2} b_{h_2} \big( \e(\omega_2 h_2 (p+1-\beta_2)) + \e(\omega_2 h_2 (p - \beta_2)) \big).
$$
\subsubsection{Estimation of $\cS_{23}$}
By Lemma~\ref{lem:index} and a splitting argument, it suffices to prove that
\begin{align}
\label{eq:S23}
&\sum_{x/2< n \le x} \sum_{0 < |h_2| \le H_2} a_{h_2} n^{\gamma-1} \Lambda(n) \nonumber \\
&\qquad \times \big( \e(\omega_2 h_2 (n+1-\beta_2)) - \e(\omega_2 h_2 (n - \beta_2)) \big) \ll x^{\gamma - \eps}.
\end{align}
 Let
\begin{equation}
\label{eq:theta2}
\theta_{h_2} \defeq \e(\omega_2 h_2) - 1.
\end{equation}
It follows from partial summation and the trivial estimate  $\theta_{h_2}\ll1$ that the left--hand side of \eqref{eq:S23} is
\begin{align*}
     &\,\, \sum_{0 < |h_2| \le H_2} a_{h_2} \sum_{x/2 < n \le x} n^{\gamma-1}
           \Lambda(n) \theta_{h_2} \e(\omega_2 h_2 (n - \beta_2)) \\
\ll  &\,\, x^{\gamma-1} \max_{x/2 < u \le x} \sum_{0 < h_2 \le H_2} h_2^{-1} \bigg|
           \sum_{x/2< n \le u} \Lambda(n) \e(\omega_2 h_2 n) \bigg|.
\end{align*}
 Hence it is sufficient to prove that
\begin{equation}
\label{eq:S23c}
 \max_{x/2 < u \le x} \sum_{0 < h_2 \le H_2} h_2^{-1} \bigg| \sum_{x/2< n \le u} \Lambda(n) \e(\omega_2 h_2 n) \bigg|
  \ll x^{1-\varepsilon}.
\end{equation}
By Lemma~\ref{lem:ttr}, the left--hand side of \eqref{eq:S23c} is
\begin{align*}
\ll & \,\, \sum_{0 < h_2 \le H_2}h_2^{-1} \Big(h_2^{1/2} x^{1-1/(2\tau) + \varepsilon}+x^{1-\varepsilon}\Big)
                \nonumber \\
\ll & \,\, \sum_{0<h_2\leqslant H_2}h_2^{-1/2}x^{1-1/(2\tau)+\varepsilon}+x^{1-\varepsilon}
            \ll x^{1-1/(2\tau)+\varepsilon}+x^{1-\varepsilon}\ll x^{1-\varepsilon}.
\end{align*}

\subsubsection{Estimation of $\cS_{24}$ and conclusions}

The contribution of $\cS_{24}$ from $h_2=0$ is
$$
\ll \sum_{p \le x} p^{\gamma-1} H_2^{-1} \ll x^{\gamma - \eps}.
$$
Hence we need the estimate that the contribution from $h_2 \neq 0$ is $\ll x^{\gamma-\varepsilon}$. Then it is sufficient to prove that
\begin{align}
\label{eq:S24}
&\sum_{x/2 < n \le x} \sum_{0 < |h_2| \le H_2} b_{h_2} n^{\gamma-1} \Lambda(n) \nonumber \\
&\qquad \times \big( \e(\omega_2 h_2 (n+1-\beta_2)) + \e(\omega_2 h_2 (n - \beta_2)) \big) \ll x^{\gamma - \eps},
\end{align}
which can be proved by the same arguments that lead to \eqref{eq:S23}. Hence we prove that $\cS_{21} \ll x^{\gamma-\eps}$. Since $\cS_{22}$ can be bounded by the same method, the estimation of $\cS_2$ is done. Moreover, the bound of $\cS_3$ can be estimated by the same proof.

\subsection{Upper bound of $\cS_4$}
In order to prove that $\cS_4 \ll x^\gamma \log^{-2} x$, we write $\cS_4 = \cS_{41} + O(\cS_{42})$, where
\begin{align*}
\cS_{41} &\defeq \sum_{p \le x} \gamma p^{\gamma-1} \big( \psi(-\omega_2(p+1-\beta_2))-\psi(-\omega_2(p-\beta_2)) \big) \\
&\qquad \times \big( \psi(-\omega_1(p+1-\beta_1))-\psi(-\omega_1(p-\beta_1)) \big),
\end{align*}
and
\begin{align*}
\cS_{42} &\defeq \sum_{p \le x} p^{\gamma-2} \big( \psi(-\omega_2(p+1-\beta_2))-\psi(-\omega_2(p-\beta_2)) \big) \\
&\qquad \times \big( \psi(-\omega_1(p+1-\beta_1))-\psi(-\omega_1(p-\beta_1)) \big).
\end{align*}
By~\eqref{eq:eta1} and~\eqref{eq:eta2}, we obtain that
$$
\cS_{41} = \cS_{43} + O(\cS_{44} + \cS_{45} + \cS_{46}),
$$
where
\begin{align*}
\cS_{43} &\defeq \sum_{p \le x} \gamma p^{\gamma-1} \sum_{0 < |h_1| \le H_1} a_{h_1} \big( \e(\omega_1 h_1 (p+1-\beta_1)) - \e(\omega_1 h_1 (p - \beta_1)) \big) \\
&\qquad \times \sum_{0 < |h_2| \le H_2} a_{h_2} \big( \e(\omega_2 h_2 (p+1-\beta_2)) - \e(\omega_2 h_2 (p - \beta_2)) \big), \\
\cS_{44} &\defeq \sum_{p \le x} \gamma p^{\gamma-1} \sum_{0 < |h_1| \le H_1} a_{h_1} \big( \e(\omega_1 h_1 (p+1-\beta_1)) - \e(\omega_1 h_1 (p - \beta_1)) \big) \\
&\qquad \times \sum_{|h_2| \le H_2} b_{h_2} \big( \e(\omega_2 h_2 (p+1-\beta_2)) + \e(\omega_2 h_2 (p - \beta_2)) \big), \\
\cS_{45} &\defeq \sum_{p \le x} \gamma p^{\gamma-1} \sum_{|h_1| \le H_1} b_{h_1} \big( \e(\omega_1 h_1 (p+1-\beta_1)) + \e(\omega_1 h_1 (p - \beta_1)) \big) \\
&\qquad \times \sum_{0 < |h_2| \le H_2} a_{h_2} \big( \e(\omega_2 h_2 (p+1-\beta_2)) - \e(\omega_2 h_2 (p - \beta_2)) \big), \\
\cS_{46} &\defeq \sum_{p \le x} \gamma p^{\gamma-1} \sum_{|h_1| \le H_1} b_{h_1} \big( \e(\omega_1 h_1 (p+1-\beta_1)) + \e(\omega_1 h_1 (p - \beta_1)) \big) \\
&\qquad \times \sum_{|h_2| \le H_2} b_{h_2} \big( \e(\omega_2 h_2 (p+1-\beta_2)) + \e(\omega_2 h_2 (p - \beta_2)) \big).
\end{align*}
\subsubsection{Estimation of $\cS_{43}$}
By Lemma~\ref{lem:index} and a splitting argument, it is sufficient to prove that
\begin{align}
\label{eq:S43}
&\sum_{x/2< n \le x} \sum_{0 < |h_1| \le H_1} \sum_{0 < |h_2| \le H_2} a_{h_1} a_{h_2} n^{\gamma-1} \Lambda(n) \nonumber \\
&\qquad \times \big( \e(\omega_1 h_1 (n+1-\beta_1)) - \e(\omega_1 h_1 (n - \beta_1)) \big) \nonumber \\
&\qquad \times \big( \e(\omega_2 h_2 (n+1-\beta_2)) - \e(\omega_2 h_2 (n - \beta_2)) \big) \ll x^{\gamma - \eps}.
\end{align}
Define
\begin{equation}
\label{eq:theta1}
\theta_{h_1} \defeq \e(\omega_1 h_1) - 1.
\end{equation}
Then by~\eqref{eq:theta2}, \eqref{eq:theta1} and the trivial estimate $\theta_{h_1}\ll1$ and $\theta_{h_2}\ll1$, we know that    the left--hand side of \eqref{eq:S43} is
\begin{align*}
&\quad \sum_{0 < |h_1| \le H_1} \sum_{0 < |h_2| \le H_2} a_{h_1} a_{h_2} \sum_{x/2< n \le x} n^{\gamma-1} \Lambda(n) \\
&\qquad\qquad\qquad\qquad\quad \times \theta_{h_1} \e(\omega_1 h_1 (n - \beta_1))
       \theta_{h_2} \e(\omega_2 h_2 (n - \beta_2)) \\
&\ll x^{\gamma-1} \max_{x/2<u\leqslant x} \sum_{0 < |h_1| \le H_1} |h_1|^{-1} \sum_{0 < |h_2| \le H_2} |h_2|^{-1}
\bigg|\sum_{x/2<n\leqslant u} \Lambda(n) \e((\omega_1 h_1 + \omega_2 h_2) n) \bigg|,
\end{align*}
which combined with Lemma \ref{lem:ttr}, Lemma \ref{lem:finite-type} and the fact that $1, \omega_1, \omega_2$ are linearly independent over $\mathbb{Q}$, yields the upper bound estimate of the left--hand side of \eqref{eq:S43}
\begin{align*}
&\ll x^{\gamma-1} \sum_{0 < h_1 \le H_1} \sum_{0 < h_2 \le H_2} h_1^{-1} h_2^{-1} x^{1-\varepsilon}  \ll x^{\gamma - 1} x^{1-\varepsilon}\ll x^{\gamma-\varepsilon}.
\end{align*}
\subsubsection{Estimations of $\cS_{44}$ and $\cS_{45}$}
We work on $\cS_{44}$ firstly. The contribution of $\cS_{44}$ from $h_2 = 0$ is
\begin{align*}
&\ll \sum_{p \le x} p^{\gamma-1} H_2^{-1} \sum_{0 < |h_1| \le H_1} a_{h_1} \big( \e(\omega_1 h_1 (p+1-\beta_1)) - \e(\omega_1 h_1 (p - \beta_1)) \big) \\
&\ll H_2^{-1} \cS_{23} \ll x^{\gamma-\eps}.
\end{align*}
The contribution of $\cS_{44}$ from $h_2 \neq 0$ is
\begin{align}\label{S_44-n_2not=0}
&\sum_{p \le x} \gamma p^{\gamma-1} \sum_{0 < |h_1| \le H_1} a_{h_1} \big( \e(\omega_1 h_1 (p+1-\beta_1)) - \e(\omega_1 h_1 (p - \beta_1)) \big) \nonumber \\
&\qquad\times \sum_{0 < |h_2| \le H_2} b_{h_2} \big( \e(\omega_2 h_2 (p+1-\beta_2)) + \e(\omega_2 h_2 (p - \beta_2)) \big).	
\end{align}
From Lemma \ref{lem:index}, it is sufficient to show that, for $x/2<u\leqslant x$, there holds
\begin{align}\label{eq:S44}
&\sum_{x/2<n\le u} n^{\gamma-1} \Lambda(n) \sum_{0 < |h_1| \le H_1} a_{h_1} \big( \e(\omega_1 h_1 (n+1-\beta_1)) - \e(\omega_1 h_1 (n - \beta_1)) \big) \nonumber \\
&\qquad \times \sum_{0 < |h_2| \le H_2} b_{h_2} \big( \e(\omega_2 h_2 (n+1-\beta_2)) + \e(\omega_2 h_2 (n - \beta_2)) \big) \ll x^{\gamma-\eps}.	
\end{align}
By the same method of bounding $\cS_{43}$, the left--hand side of~\eqref{eq:S44} is
\begin{align*}
&\ll x^{\gamma-1} \max_{x/2<u\leqslant x} \sum_{0 < |h_1| \le H_1} |h_1|^{-1} \sum_{0 < |h_2| \le H_2} H_2^{-1}
\bigg| \sum_{x/2<n\le u} \Lambda(n) \e((\omega_1 h_1 + \omega_2 h_2) n) \bigg| \\
&\ll x^{\gamma - 1}\cdot\big(x^{1-1/(2\tau)+\varepsilon}+x^{1-\varepsilon}\big)\ll x^{\gamma-\varepsilon}.
\end{align*}
Similarly, we can prove that $\cS_{45} \ll x^{\gamma-\eps}$.
\subsubsection{Estimation of $\cS_{46}$}
The contribution from $h_1 = h_2 = 0$ is
$$
\ll \sum_{p \le x} p^{\gamma-1} H_1^{-1} H_2^{-1} \ll x^{\gamma - \eps}.
$$
The contribution from $h_1 = 0$ and $h_2 \neq 0$ is
\begin{align*}
&\ll \sum_{p \le x} p^{\gamma-1} H_1^{-1} \sum_{|h_2| \le H_2} b_{h_2} \big( \e(\omega_2 h_2 (p+1-\beta_2)) + \e(\omega_2 h_2 (p - \beta_2)) \big) \\
&\ll H_1^{-1} \cS_{24} \ll x^{\gamma-\eps}.
\end{align*}
The contribution from $h_1 \neq 0$ and $h_2 = 0$ can be bounded by a similar method. In the end, the contribution from $h_1 \neq 0$ and $h_2 \neq 0$ is
\begin{align*}
&\sum_{p \le x} \gamma p^{\gamma-1} \sum_{0 < |h_1| \le H_1} b_{h_1} \big( \e(\omega_1 h_1 (p+1-\beta_1)) + \e(\omega_1 h_1 (p - \beta_1)) \big) \\
&\qquad \times \sum_{0 < |h_2| \le H_2} b_{h_2} \big( \e(\omega_2 h_2 (p+1-\beta_2)) + \e(\omega_2 h_2 (p - \beta_2)) \big),	
\end{align*}
which can be bounded by using the same method of \eqref{S_44-n_2not=0}.

\subsection{Upper bounds of $\cS_5$ and $\cS_6$}
We only give the details of the estimation of $\cS_5$, since the estimate of $\cS_6$ is exactly the same as that of $\cS_5$, and we omit it herein. By~\eqref{eq:eta2} and~\eqref{eq:gamma}, it is easy to see that
$$
\cS_5 = \cS_{51} + O(\cS_{52}+ \cS_{53} + \cS_{54}),
$$
where
\begin{align*}
\cS_{51} &\defeq \sum_{p \le x} \omega_1 \sum_{0 < |h_3| \le H_3} a_{h_3} \big( \e(h_3 (p+1)^\gamma) - \e(h_3 p^\gamma ) \big) \\
&\qquad \times \sum_{0 < |h_2| \le H_2} a_{h_2} \big( \e(\omega_2 h_2 (p+1-\beta_2)) - \e(\omega_2 h_2 (p - \beta_2)) \big), \\
\cS_{52} &\defeq \sum_{p \le x} \omega_1 \sum_{0 < |h_3| \le H_3} a_{h_3} \big( \e(h_3 (p+1)^\gamma) - \e(h_3 p^\gamma ) \big) \\
&\qquad \times \sum_{|h_2| \le H_2} b_{h_2} \big( \e(\omega_2 h_2 (p+1-\beta_2)) + \e(\omega_2 h_2 (p - \beta_2)) \big), \\
\cS_{53} &\defeq \sum_{p \le x} \omega_1 \sum_{|h_3| \le H_3} b_{h_3} \big( \e(h_3 (p+1)^\gamma) + \e(h_3 p^\gamma ) \big) \\
&\qquad \times \sum_{0 < |h_2| \le H_2} a_{h_2} \big( \e(\omega_2 h_2 (p+1-\beta_2)) - \e(\omega_2 h_2 (p - \beta_2)) \big), \\
\cS_{54} &\defeq \sum_{p \le x} \omega_1 \sum_{|h_3| \le H_3} b_{h_3} \big( \e(h_3 (p+1)^\gamma) + \e(h_3 p^\gamma ) \big) \\
&\qquad \times \sum_{|h_2| \le H_2} b_{h_2} \big( \e(\omega_2 h_2 (p+1-\beta_2)) + \e(\omega_2 h_2 (p - \beta_2)) \big).
\end{align*}
\subsubsection{Estimation of $\cS_{51}$}
In order to show that $\cS_{51} \ll x^{\gamma-\eps}$, by Lemma~\ref{lem:index} and a splitting argument,
it suffices to prove that
\begin{align}
\label{eq:S51}
&\sum_{x/2< n \le x} \Lambda(n) \sum_{0 < |h_3| \le H_3} a_{h_3} \big( \e(h_3 (n+1)^\gamma) - \e(h_3 n^\gamma ) \big) \nonumber \\
&\qquad \times \sum_{0 < |h_2| \le H_2} a_{h_2} \big( \e(\omega_2 h_2 (n+1-\beta_2)) - \e(\omega_2 h_2 (n - \beta_2)) \big) \ll x^{\gamma - \eps}.
\end{align}
Define
\begin{equation*}
\phi_{h_3}(t):=\e\big(h_3((t+1)^\gamma-t^\gamma)\big)-1.
\end{equation*}
Then we have
\begin{equation*}
\phi_{h_3}(t)\ll |h_3| t^{\gamma -1} \mand \frac{\partial\phi_{h_3}(t)}{\partial t}\ll|h_3|t^{\gamma-2}.
\end{equation*}
It follows from the above estimate, \eqref{eq:theta2}, lemma \ref{lem:exp} and partial summation that the left--hand side of \eqref{eq:S51} is
\begin{align*}
\ll & \,\, \sum_{0<|h_3|\leqslant H_3} \frac{1}{|h_3|}\Bigg|\sum_{x/2<n\leqslant x}\Lambda(n)\phi_{h_3}(n)\e(h_3n^\gamma)
                  \nonumber \\
    & \,\, \qquad \qquad\times \sum_{0<|h_2|\leqslant H_2}a_{h_2}\big(\e(\omega_2 h_2 (n+1-\beta_2))-
            \e(\omega_2 h_2 (n - \beta_2)) \big)\Bigg|
                  \nonumber \\
\ll & \,\, \sum_{0<|h_3|\leqslant H_3} \frac{1}{|h_3|}\Bigg|\int_{\frac{x}{2}}^x\phi_{h_3}(t)
           \mathrm{d}\Bigg(\sum_{x/2<n\leqslant t}\Lambda(n)\e(h_3n^\gamma)
                  \nonumber \\
   & \,\,\qquad \qquad\times \sum_{0<|h_2|\leqslant H_2}a_{h_2}\big(\e(\omega_2 h_2 (n+1-\beta_2))-\e(\omega_2h_2(n-\beta_2))\big)\Bigg)\Bigg|
                  \nonumber \\
\ll & \,\, \sum_{0<|h_3|\leqslant H_3} \frac{1}{|h_3|}\Big|\phi_{h_3}(x)\Big|\Bigg|\sum_{x/2<n\leqslant x}
           \Lambda(n)\e(h_3n^\gamma)
                  \nonumber \\
    & \,\,\qquad \qquad\times \sum_{0<|h_2|\leqslant H_2}
                  a_{h_2}\big(\e(\omega_2h_2(n+1-\beta_2))-\e(\omega_2h_2(n-\beta_2))\big)\Bigg|
                  \nonumber \\
    & \,\, +\int_{\frac{x}{2}}^x\sum_{0<|h_3|\leqslant H_3} \frac{1}{|h_3|}\Bigg|\frac{\partial\phi_{h_3}(t)}{\partial t}\Bigg|
           \Bigg|\sum_{x/2<n\leqslant t}\Lambda(n)\e(h_3n^\gamma)
                  \nonumber \\
   & \,\, \qquad \qquad\times\sum_{0<|h_2|\leqslant H_2}
                 a_{h_2}\big(\e(\omega_2h_2(n+1-\beta_2))-\e(\omega_2h_2(n-\beta_2))\big)
                \Bigg|\mathrm{d}t
                  \nonumber \\
\ll & \,\, x^{\gamma-1}\cdot\max_{x/2<t\leqslant x}\sum_{0<|h_3|\leqslant H_3}\Bigg|\sum_{x/2<n\leqslant t}
           \Lambda(n)\e(h_3n^\gamma)
                  \nonumber \\
    & \,\,\qquad \qquad\times \sum_{0<|h_2|\leqslant H_2}
                 a_{h_2}\big(\e(\omega_2h_2(n+1-\beta_2))-\e(\omega_2h_2(n-\beta_2))\big)\Bigg|
                  \nonumber \\
 = & \,\,  x^{\gamma-1}\cdot\max_{x/2<t\leqslant x}\sum_{0<|h_3|\leqslant H_3}\Bigg|
           \sum_{0<|h_2|\leqslant H_2}a_{h_2}\theta_{h_2}
                  \nonumber \\
  & \,\,\qquad \qquad \times \sum_{x/2<n\leqslant t}\Lambda(n)\e(h_3n^\gamma+\omega_2h_2n-\omega_2h_2\beta_2)\Bigg|
                  \nonumber \\
 \ll & \,\,  x^{\gamma-1}\cdot\max_{x/2<t\leqslant x}\sum_{0<|h_2|\leqslant H_2}\frac{1}{|h_2|}\sum_{0<|h_3|\leqslant H_3}
             \Bigg|\sum_{x/2<n\leqslant t}\Lambda(n)\e(h_3n^\gamma+\omega_2h_2n-\omega_2h_2\beta_2)\Bigg|
                  \nonumber \\
\ll & \,\,x^{\gamma-1}\sum_{0<|h_2|\leqslant H_2}\frac{1}{|h_2|}\sum_{0<|h_3|\leqslant H_3}\max_{x/2<t\leqslant x}
             \Bigg|\sum_{x/2<n\leqslant t}\Lambda(n)\e(h_3n^\gamma+\omega_2h_2n-\omega_2h_2\beta_2)\Bigg|
                  \nonumber \\
\ll & \,\,x^{\gamma-1+\varepsilon}\sum_{0<|h_2|\leqslant H_2}\frac{1}{|h_2|}\sum_{0<|h_3|\leqslant H_3}
          \Big(|h_3|^{1/6}x^{\gamma/6+3/4}+|h_3|^{-1/3}x^{1-\gamma/3} \\
    &\qquad\qquad  + |h_3|^{1/4} x^{\gamma/4+5/8} + |h_3|^{-1/4}x^{1-\gamma/4} + x^{22/25} \Big)
                 \nonumber \\
\ll & \,\, x^{\gamma-1+\varepsilon}\Big(H_3^{7/6}x^{\gamma/6+3/4}+H_3^{2/3}x^{1-\gamma/3}+H_3^{5/4}x^{\gamma/4+5/8} +
           H_3^{3/4}x^{1-\gamma/4} + H_3 x^{22/25} \Big)
                 \nonumber  \\
\ll & \,\, x^{11/12+\varepsilon}+x^{2/3+\varepsilon}+x^{7/8+\varepsilon}+x^{3/4+\varepsilon}+x^{22/25+\varepsilon}
           \ll x^{\gamma-\varepsilon},
\end{align*}
provided that $\gamma > \frac{11}{12}$.

\subsubsection{Estimations of $\cS_{52}$ and $\cS_{53}$}
We only give the proof of $\cS_{52}$, since the bound of $\cS_{53}$ can be obtained similarly. We mention that by assuming $\gamma > \frac{11}{12}$, there holds
\begin{equation}\label{eq:GKa}
\sum_{p\le x}\sum_{0<|h_3|\le H_3}a_{h_3}\big(\e(h_3(p+1)^\gamma)-\e(h_3p^\gamma)\big)\ll x^{\gamma}\log^{-2} x,
\end{equation}
by a standard proof of the Piatetski--Shapiro counting function. A detailed proof can be found in~\cite[pp. 49--53]{GraKol}. By~\eqref{eq:GKa} the contribution of $\cS_{52}$ from $h_2 = 0$ is
\begin{align*}
&\ll \sum_{p \le x} H_2^{-1} \sum_{0 < |h_3| \le H_3} a_{h_3} \big( \e(h_3 (p+1)^\gamma) - \e(h_3 p^\gamma ) \big) \\
&\ll H_2^{-1} x^\gamma \log^{-2} x \ll x^{\gamma - \eps}.
\end{align*}
The contribution of $\cS_{52}$ from $h_2 \neq 0$ is
\begin{align*}
&\sum_{p \le x} \sum_{0 < |h_3| \le H_3} a_{h_3} \big( \e(h_3 (p+1)^\gamma) - \e(h_3 p^\gamma ) \big) \\
&\qquad \times \sum_{0 < |h_2| \le H_2} b_{h_2} \big( \e(\omega_2 h_2 (p+1-\beta_2)) + \e(\omega_2 h_2 (p - \beta_2)) \big),
\end{align*}
which can be bounded by the same method of~\eqref{eq:S51}.
\subsubsection{Estimation of $\cS_{54}$}
The contribution of $\cS_{54}$ from $h_2 = h_3 = 0$ is
$$
\ll\sum_{p\le x}H_3^{-1}H_2^{-1}\ll \frac{x}{\log x}\cdot x^{-(1-\gamma+\varepsilon)}\cdot x^{-\varepsilon}
                \ll x^{\gamma-\varepsilon}.
$$
We mention that
\begin{equation}\label{eq:GKb}
\sum_{p \le x} \sum_{0 < |h_3| \le H_3} b_{h_3} \big( \e(h_3 (p+1)^\gamma) + \e(h_3 p^\gamma ) \big) \ll x^{\gamma} \log^{-2} x,
\end{equation}
by a standard method of exponent pair; see~\cite[pp. 48]{GraKol}. By~\eqref{eq:GKb} the contribution of $\cS_{54}$ from $h_2 = 0$ and $h_3 \neq 0$ is
\begin{align*}
&\ll \sum_{p \le x} H_2^{-1} \sum_{0 < |h_3| \le H_3} b_{h_3} \big( \e(h_3 (p+1)^\gamma) + \e(h_3 p^\gamma ) \big) \\
&\ll H_2^{-1} x^\gamma \log^{-2} x \ll x^{\gamma - \eps}.
\end{align*}
Similarly, the contribution of $\cS_{54}$ from $h_2 \neq 0$ and $h_3 = 0$ is
\begin{align*}
&\ll \sum_{p \le x} H_3^{-1} \sum_{|h_2| \le H_2} b_{h_2} \big( \e(\omega_2 h_2 (p+1-\beta_2)) + \e(\omega_2 h_2 (p - \beta_2)) \big) \\
&\ll H_3^{-1}\sum_{p\leqslant x}1 \ll x^{\gamma-1-\varepsilon}\cdot\frac{x}{\log x}\ll x^{\gamma-\varepsilon}.
\end{align*}
The contribution of $\cS_{54}$ from $h_2 \neq 0$ and $h_3 \neq 0$ is
\begin{align*}
\ll &\,\, \sum_{n \le x}\Lambda(n)\sum_{0<|h_3|\le H_3} b_{h_3} \big( \e(h_3 (n+1)^\gamma) + \e(h_3n^\gamma )\big) \\
&\qquad \times \sum_{0<|h_2|\leqslant H_2}b_{h_2}\big(\e(\omega_2h_2(n+1-\beta_2))+\e(\omega_2h_2(n-\beta_2))\big),
\end{align*}
which can be estimated by the method demonstrated on page 48 in \cite{GraKol} to derive the
upper bound $x^{\gamma-\varepsilon}$.

\subsection{Upper bound of $\cS_7$}
By \eqref{eq:eta1}, \eqref{eq:eta2} and \eqref{eq:gamma}, we write
$$
\cS_7 = \cS_{71} + O(\cS_{72} + \cS_{73} + \cS_{74} + \cS_{75} + \cS_{76} + \cS_{77} + \cS_{78}),
$$
where
\begin{align*}
\cS_{71} &\defeq \sum_{p \le x} \sum_{0 < |h_1| \le H_1} a_{h_1} \big( \e(\omega_1 h_1 (p+1-\beta_1)) - \e(\omega_1 h_1 (p - \beta_1)) \big) \\
&\qquad \times \sum_{0 < |h_2| \le H_2} a_{h_2} \big( \e(\omega_2 h_2 (p+1-\beta_2)) - \e(\omega_2 h_2 (p - \beta_2)) \big) \\
&\qquad \times \sum_{0 < |h_3| \le H_3} a_{h_3} \big( \e(h_3 (p+1)^\gamma) - \e(h_3 p^\gamma ) \big), \\
\cS_{72} &\defeq \sum_{p \le x} \sum_{0 < |h_1| \le H_1} a_{h_1} \big( \e(\omega_1 h_1 (p+1-\beta_1)) - \e(\omega_1 h_1 (p - \beta_1)) \big) \\
&\qquad \times \sum_{0 < |h_2| \le H_2} a_{h_2} \big( \e(\omega_2 h_2 (p+1-\beta_2)) - \e(\omega_2 h_2 (p - \beta_2)) \big) \\
&\qquad \times \sum_{|h_3| \le H_3} b_{h_3} \big( \e(h_3 (p+1)^\gamma) + \e(h_3 p^\gamma ) \big), \\
\cS_{73} &\defeq \sum_{p \le x} \sum_{0 < |h_1| \le H_1} a_{h_1} \big( \e(\omega_1 h_1 (p+1-\beta_1)) - \e(\omega_1 h_1 (p - \beta_1)) \big) \\
&\qquad \times \sum_{|h_2| \le H_2} b_{h_2} \big( \e(\omega_2 h_2 (p+1-\beta_2)) + \e(\omega_2 h_2 (p - \beta_2)) \big) \\
&\qquad \times \sum_{0 < |h_3| \le H_3} a_{h_3} \big( \e(h_3 (p+1)^\gamma) - \e(h_3 p^\gamma ) \big), \\
\cS_{74} &\defeq \sum_{p \le x} \sum_{0 < |h_1| \le H_1} a_{h_1} \big( \e(\omega_1 h_1 (p+1-\beta_1)) - \e(\omega_1 h_1 (p - \beta_1)) \big) \\
&\qquad \times \sum_{|h_2| \le H_2} b_{h_2} \big( \e(\omega_2 h_2 (p+1-\beta_2)) + \e(\omega_2 h_2 (p - \beta_2)) \big) \\
&\qquad \times \sum_{|h_3| \le H_3} b_{h_3} \big( \e(h_3 (p+1)^\gamma) + \e(h_3 p^\gamma ) \big), \\
\cS_{75} &\defeq \sum_{p \le x} \sum_{|h_1| \le H_1} b_{h_1} \big( \e(\omega_1 h_1 (p+1-\beta_1)) + \e(\omega_1 h_1 (p - \beta_1)) \big) \\
&\qquad \times \sum_{0 < |h_2| \le H_2} a_{h_2} \big( \e(\omega_2 h_2 (p+1-\beta_2)) - \e(\omega_2 h_2 (p - \beta_2)) \big) \\
&\qquad \times \sum_{0 < |h_3| \le H_3} a_{h_3} \big( \e(h_3 (p+1)^\gamma) - \e(h_3 p^\gamma ) \big), \\
\cS_{76} &\defeq \sum_{p \le x} \sum_{|h_1| \le H_1} b_{h_1} \big( \e(\omega_1 h_1 (p+1-\beta_1)) + \e(\omega_1 h_1 (p - \beta_1)) \big) \\
&\qquad \times \sum_{0 < |h_2| \le H_2} a_{h_2} \big( \e(\omega_2 h_2 (p+1-\beta_2)) - \e(\omega_2 h_2 (p - \beta_2)) \big) \\
&\qquad \times \sum_{|h_3| \le H_3} b_{h_3} \big( \e(h_3 (p+1)^\gamma) + \e(h_3 p^\gamma ) \big), \\
\cS_{77} &\defeq \sum_{p \le x} \sum_{|h_1| \le H_1} b_{h_1} \big( \e(\omega_1 h_1 (p+1-\beta_1)) + \e(\omega_1 h_1 (p - \beta_1)) \big) \\
&\qquad \times \sum_{|h_2| \le H_2} b_{h_2} \big( \e(\omega_2 h_2 (p+1-\beta_2)) + \e(\omega_2 h_2 (p - \beta_2)) \big) \\
&\qquad \times \sum_{0 < |h_3| \le H_3} a_{h_3} \big( \e(h_3 (p+1)^\gamma) - \e(h_3 p^\gamma ) \big), \\
\cS_{78} &\defeq \sum_{p \le x} \sum_{|h_1| \le H_1} b_{h_1} \big( \e(\omega_1 h_1 (p+1-\beta_1)) + \e(\omega_1 h_1 (p - \beta_1)) \big) \\
&\qquad \times \sum_{|h_2| \le H_2} b_{h_2} \big( \e(\omega_2 h_2 (p+1-\beta_2)) + \e(\omega_2 h_2 (p - \beta_2)) \big) \\
&\qquad \times \sum_{|h_3| \le H_3} b_{h_3} \big( \e(h_3 (p+1)^\gamma) + \e(h_3 p^\gamma ) \big).
\end{align*}
\subsubsection{Estimation of $\cS_{71}$}
In order to show that $\cS_{71} \ll x^{\gamma-\eps}$, we apply the similar argument which derives the upper bound of $\cS_{51}$. Therefore, it is sufficient to prove that
\begin{equation}\label{eq:S71}
\max_{x/2<t\leqslant x}\sum_{0<|h_1|\le H_1}|h_1|^{-1}\sum_{0<|h_2|\le H_2}|h_2|^{-1}
\sum_{0<h_3\le H_3}|\cT|\ll x^{1-\varepsilon},	
\end{equation}
where
$$
\cT\defeq\sum_{x/2<n\le t}\Lambda(n)\cdot\e\big(h_3 n^\gamma+(\omega_1 h_1+\omega_2h_2)n
                    -\omega_1h_1\beta_1-\omega_2h_2\beta_2\big),
$$
which can be bounded by Lemma~\ref{lem:exp}. By a similar argument of the estimation of the left--hand side of \eqref{eq:S51}, we derive \eqref{eq:S71} by assuming that $\gamma > \frac{11}{12}$.
\subsubsection{Estimation of $\cS_{72}$}
The contribution from $h_3 = 0$ is
\begin{align*}
&\sum_{p \le x} H_3^{-1} \sum_{0 < |h_1| \le H_1} a_{h_1} \big( \e(\omega_1 h_1 (p+1-\beta_1)) - \e(\omega_1 h_1 (p - \beta_1)) \big) \\
&\qquad \times \sum_{0 < |h_2| \le H_2} a_{h_2} \big( \e(\omega_2 h_2 (p+1-\beta_2)) - \e(\omega_2 h_2 (p - \beta_2)) \big) \\
&\ll H_3^{-1} x^{1-\eps} \ll x^{\gamma - \eps},	
\end{align*}
by the estimation of $\cS_{43}$. The contribution from $h_3 \neq 0$ can be bounded by the same method of $\cS_{71}$.
\subsubsection{Estimations of $\cS_{73}$ and $\cS_{75}$}
The contribution of $\cS_{73}$ from $h_2 = 0$ is
\begin{align*}
&\sum_{p \le x} H_2^{-1} \sum_{0 < |h_1| \le H_1} a_{h_1} \big( \e(\omega_1 h_1 (p+1-\beta_1)) - \e(\omega_1 h_1 (p - \beta_1)) \big) \\
&\qquad \times \sum_{0 < |h_3| \le H_3} a_{h_3} \big( \e(h_3 (p+1)^\gamma) - \e(h_3 p^\gamma ) \big),
\end{align*}
which can be bounded by the method of $\cS_{51}$. The contribution of $\cS_{73}$ from $h_2 \neq 0$ can be bounded by the
similar arguments which deal with the upper bound of $\cS_{51}$. The estimation of $\cS_{75}$ is exactly the same as $\cS_{72}$.

\subsubsection{Estimations of $\cS_{74}$ and $\cS_{76}$}
The contribution of $\cS_{74}$ from $h_2 = h_3 = 0$ is
\begin{align*}
&\quad \sum_{p \le x} H_2^{-1} H_3^{-1} \sum_{0 < |h_1| \le H_1} a_{h_1} \big( \e(\omega_1 h_1 (p+1-\beta_1)) - \e(\omega_1 h_1 (p - \beta_1)) \big) \\
&\ll H_2^{-1} H_3^{-1} x^{1-\eps} \ll x^{\gamma - \eps},
\end{align*}
by the same method of $\cS_{23}$. The contribution of $\cS_{74}$ from $h_2 = 0$ and $h_3 \neq 0$ is
\begin{align*}
&\sum_{p \le x} H_2^{-1} \sum_{0 < |h_1| \le H_1} a_{h_1} \big( \e(\omega_1 h_1 (p+1-\beta_1)) - \e(\omega_1 h_1 (p - \beta_1)) \big) \\
&\qquad \times \sum_{0 < |h_3| \le H_3} b_{h_3} \big( \e(h_3 (p+1)^\gamma) + \e(h_3 p^\gamma ) \big),	
\end{align*}
which can be bounded by the similar method of $\cS_{51}$. Similarly, the contribution of $\cS_{74}$ from $h_2 \neq 0$ and $h_3 = 0$ is bounded by the same method of $\cS_{44}$. The contribution of $\cS_{74}$ from $h_2 \neq 0$ and $h_3 \neq 0$ can be
bounded by the similar arguments which deal with the upper bound of $\cS_{23}$. The estimation of $\cS_{76}$ is almost the same as $\cS_{74}$.

\subsubsection{Estimation of $\cS_{77}$}
The contribution from $h_1 = h_2 = 0$ can be bounded from the estimate (\ref{eq:GKa}). The contribution from $h_1 = 0$ and $h_2 \neq 0$ is bounded by the same method of $\cS_{43}$ and similar to the contribution from $h_1 \neq 0$ and $h_2 = 0$. The contribution from $h_1 \neq 0$ and $h_2 \neq 0$ can be bounded by following the similar argument of $\cS_{52}$.

\subsubsection{Estimation of $\cS_{78}$}
The contribution from $h_1 = h_2 = h_3 = 0$ is
$$
\ll \sum_{p \le x} H_1^{-1} H_2^{-1} H_3^{-1} \ll x^{\gamma - \eps}.
$$
The contribution from $h_1 \neq 0$ and $h_2 = h_3 = 0$ is bounded by same method of $\cS_{24}$ and similar to the contribution from $h_2 \neq 0$ and $h_1 = h_3 = 0$. The contribution from $h_1 \neq 0$, $h_2 \neq 0$ and $h_3 = 0$ is bounded by the same method of $\cS_{43}$. The contribution from $h_1 = h_2 = 0$ and $h_3 \neq 0$ can be bounded by following the argument on page
48 of \cite{GraKol}. The contribution from $h_1 \neq 0$, $h_2 = 0$ and $h_3 \neq 0$ is bounded by the same method of $\cS_{51}$ and similar to the contribution from $h_1 = 0$, $h_2 \neq 0$ and $h_3 \neq 0$. In the end, the contribution from $h_1 \neq 0$, $h_2 \neq 0$ and $h_3 \neq 0$ is bounded by the same method of $\cS_{71}$.

\section{Proof of Theorem 2}
For a Beatty sequence
$$
\cB_{\alpha, \beta} \defeq \fl{\alpha n + \beta},
$$
recall that $\omega \defeq \alpha^{-1}$. Similar to the construction of the proof of Theorem~\ref{thm:main} and by the definition of $\pi^{(c)}_{\alpha, \beta} (x)$, we have that
$$
\pi^{(c)}_{\alpha, \beta} (x) = \sum_{p \le x} \cX_{\alpha, \beta}(p) \cX^{(c)}(p) = S_1 + S_2 + S_3,
$$
where
\begin{align*}
S_1 &\defeq \sum_{p \le x} \omega \cX^{(c)}(p); \\
S_2 &\defeq \sum_{p \le x} \big(  \gamma p^{\gamma-1}+O(p^{\gamma-2}) \big) \\
&\qquad \times \big( \psi(-\omega(p+1-\beta))-\psi(-\omega(p-\beta)) \big), \\
S_3 &\defeq \sum_{p \le x} \big( \psi(-(p+1)^\gamma)-\psi(-p^\gamma) \big) \\
&\qquad \times \big( \psi(-\omega(p+1-\beta))-\psi(-\omega(p-\beta)) \big).
\end{align*}
$S_1$ can be estimated by the same method of $\cS_1$ in the proof of Theorem~\ref{thm:main}, which is
$$
S_1 = \frac{x^\gamma}{\alpha \log x} + O\( \frac{x}{\log^2 x} \).
$$
$S_2$ can be bounded by the same method of $\cS_2$ in the proof of Theorem~\ref{thm:main}. By the assumption that $\alpha$ is of finite type, it gives that
$$
S_2 \ll x^{\gamma - \eps}.
$$
$S_3$ can be estimated by the same method of $\cS_5$ in the proof of Theorem~\ref{thm:main}, which gives that
$$
S_3 \ll x^{\gamma - \eps},
$$
provided that $\gamma > \frac{11}{12}$.

\section{Proof of Theorem 3}
The method is similar to the proof of Theorem~\ref{thm:main} with more technical summations since there are more characteristic functions of Beatty sequences. Therefore, we only give a sketch proof here.
Let $\omega_i \defeq \alpha_i^{-1}$ for $i \in \{ 1, \dots, \xi \}$. By the similar construction, we have that
\begin{equation}\label{eq:mul}
\pi^{(c)}_{\alpha_1,\beta_1; \dots; \alpha_{\xi},\beta_{\xi}} (x) = \sum_{p \le x} \cX^{(c)}(p) \prod_{i=1}^{\xi} \cX_{\alpha_i, \beta_i} (p),
\end{equation}
and
$$
\cX_{\alpha_i, \beta_i} = \omega_i + \psi(-\omega_i (p+1-\beta_i))-\psi(-\omega_i(p-\beta_i)).
$$
We can break~\eqref{eq:mul} into several sums by the Vaaler's approximation, which is similar to the construction of the proof of Theorem~\ref{thm:main}. Since every sum with $b_h$ can be bounded by separating the contribution of $h = 0 $ and $h \neq 0$ and compared with the corresponding sum with $a_h$, we only give a sketch of the estimations of the summations with $a_h$.

The main term of~\eqref{eq:mul} is
$$
\sum_{p \le x} \omega_1 \cdots \omega_{\xi} \cX^{(c)}(p) = \frac{x^\gamma}{\alpha_1 \cdots \alpha_{\xi} \log x} + O\(\frac{x^\gamma}{\log^2 x}\),
$$
with $c\in(1,\frac{2817}{2426})$ by~\eqref{eq:PSthm}. Let $\sL$ be an arbitrary subset of $\{1, \dots, \xi\}$. Set
$$
H_1 = \cdots = H_{\xi} \defeq x^\eps \mand J \defeq x^{1-\gamma + \eps}.
$$
We claim that
$$
\pi^{(c)}_{\alpha_1,\beta_1; \dots; \alpha_{\xi},\beta_{\xi}} (x) = \frac{x^\gamma}{\alpha_1 \cdots \alpha_{\xi} \log x} + O\(\frac{x^\gamma}{\log^2 x} + \cT_1 + \cT_2\),
$$
where
$$
\cT_1\defeq\sum_{p\le x}p^{\gamma-1}\prod_{i\in\sL}\sum_{0<|h_i|\le H_i}a_{h_i}\big(\e(\omega_ih_i(p+1-\beta_i))-
\e(\omega_ih_i(p-\beta_i))\big)
$$
and
\begin{align*}
\cT_2 &\defeq \sum_{p \le x} \Bigg( \prod_{i \in \sL} \sum_{0 < |h_i| \le H_i} a_{h_i} \big( \e(\omega_i h_i (p+1-\beta_i)) - \e(\omega_i h_i (p - \beta_i)) \big) \\
&\qquad\qquad \times \sum_{0 < |j| \le J} a_{j} \big( \e(j (p+1)^\gamma) - \e(j p^\gamma ) \big) \Bigg).
\end{align*}	
We consider $\cT_1$ firstly. By the same method of $\cS_{43}$, in order to show that $\cT_1 \ll x^{\gamma - \eps}$, it is sufficient to prove that
\begin{equation}\label{eq:cT1}
\max_{x/2<u\leqslant x}\prod_{i \in \sL}\sum_{0<|h_i|\le H_i}|h_i|^{-1}|\cT_3|\ll x^{1-\varepsilon},
\end{equation}
where
$$
\cT_3 \defeq \sum_{x/2<n\leqslant u}\Lambda(n)\e\(\(\sum_{i \in \sL}\omega_ih_i\)n\).
$$
By Lemma~\ref{lem:ttr} and Lemma \ref{lem:finite-type}, the left--hand side of \eqref{eq:cT1} is
$$
\ll\(\prod_{i\in\sL}\sum_{0<|h_i|\le H_i}|h_i|^{-1}\) \Big(x^{1-1/(2\tau)+\varepsilon}+x^{1-\varepsilon}\Big)
\ll x^{1-\varepsilon},
$$
provided that $\alpha_i$ is of finite type with $i\in\{1,2,\dots,\xi\}$.

Now we consider $\cT_2$. By the similar method of $\cS_{71}$, it is sufficient to prove that
\begin{equation}\label{eq:cT2}
\max_{x/2<t\leqslant x} \prod_{i \in \sL} \sum_{0 < |h_i| \le H_i} |h_i|^{-1} \sum_{0 < j \le J} |\cT_4| \ll x^{1-\varepsilon}, \end{equation}
where
$$
\cT_4\defeq\sum_{x/2<n\le t}\Lambda(n)\cdot\e\(jn^\gamma + \sum_{i\in\sL} (\omega_ih_in-\omega_ih_i\beta_i) \),
$$
which can be bounded by Lemma~\ref{lem:exp}. By a similar argument of the estimation of the left--hand side of \eqref{eq:S51}, we derive \eqref{eq:cT2} under the condition that $\gamma > \frac{11}{12}$.

\section{Acknowledgement}
The first author is supported by the National Natural Science Foundation of China (Grant No. 11901447), the
China Postdoctoral Science Foundation (Grant No. 2019M653576), the Natural Science Foundation of Shaanxi
Province (Grant No. 2020JQ009), and the China Scholarship Council (Grant No. 201906285058).

The second author (Corresponding author) is supported by the National Natural Science Foundation of China
(Grant No. 11901566, 11971476, 12071238), the Fundamental Research Funds for the Central Universities
(Grant No. 2021YQLX02), and National Training Program of Innovation and Entrepreneurship for Undergraduates
(Grant No. 202107010).

The third author is supported by the National Natural Science Foundation of China (Grant No. 12001047, 11971476),
and the Scientific Research Funds of Beijing Information Science and Technology University (Grant No. 2025035).

\end{document}